\documentclass{amsart}
\usepackage{amssymb,amsmath, color}

\usepackage{amssymb}

\newtheorem{Theorem}{Theorem}[section]
\newtheorem{lemma}{Lemma}[section]
\newtheorem{corollaire}{Corollary}[section]

\newtheorem{proposition}{Proposition}[section]

\theoremstyle{remark}
\newtheorem{remark}{Remark}[section]

\newcommand{\R}{{\mathbb R}}
\newcommand{\N}{{\mathbb N}}
\newcommand{\argmin}{{\rm argmin}\kern 0.12em}

\usepackage{ulem}

\newcommand{\dta}{\frac{2}{3}\alpha}
\newcommand{\prox}{{\rm prox}\kern 0.06em}
\newcommand{\cH}{{\mathcal H}}
\newcommand{\rinf}{\mathbb R\cup\{+\infty\}}
\newcommand{\demi}{\frac{1}{2}}

\begin{document}

\title{Rate of convergence  of the Nesterov accelerated gradient method\\ in the subcritical case  $\alpha \leq 3$.}

\author{Hedy Attouch}
\address{Institut Montpelli\'erain A. Grothendieck, UMR CNRS 5149, Universit\'e Montpellier, 
34095 Montpellier cedex 5, France}
\email{hedy.attouch@univ-montp2.fr}

\author{Zaki Chbani}\address{Cadi Ayyad university, Faculty of Sciences Semlalia, Mathematics, 40000 Marrakech, Morroco}
\email{chbaniz@uca.ac.ma}

\author{Hassan Riahi}
\address{Cadi Ayyad university, Faculty of Sciences Semlalia, Mathematics, 40000 Marrakech, Morroco}
\email{h-riahi@uca.ac.ma}

\date{June 18, 2017}

\begin{abstract}
In a Hilbert space setting $\mathcal H$, given $\Phi: \mathcal H \to  \mathbb R$ a convex continuously differentiable function, and $\alpha$ a positive parameter, we consider  the inertial system with Asymptotic Vanishing Damping
\begin{equation*}
 \mbox{(AVD)}_{\alpha} \quad \quad \ddot{x}(t) + \frac{\alpha}{t} \dot{x}(t) + \nabla \Phi (x(t)) =0.
\end{equation*}
 Depending on the value of $ \alpha $ with respect to 3, we give a complete picture of the  convergence properties as $t \to + \infty$ of the trajectories generated by $\mbox{(AVD)}_{\alpha}$, as well as  iterations of the corresponding algorithms.
Indeed, as shown by  Su-Boyd-Cand\`es,  the case $\alpha = 3$ corresponds to a continuous version of the accelerated gradient method of  Nesterov, with the   rate of convergence
$\Phi (x(t))-\min \Phi =   \mathcal O (t^{-2})$ for $\alpha \geq 3$. 
Our main result concerns the subcritical case $\alpha \leq 3$, where we show  that
$\Phi (x(t))-\min \Phi =   \mathcal O (t^{-\frac{2}{3}\alpha})$.  
This overall picture shows a continuous variation of the rate of convergence of the values $\Phi(x(t))-\min_\mathcal H \Phi= \mathcal O (t^{-p(\alpha)})  $ with respect to $\alpha >0$: the coefficient $p(\alpha)$ increases linearly up to $2$ when $\alpha$ goes from $0$ to $3$,  then displays a plateau. Then we examine the convergence of trajectories to optimal solutions. When  $\alpha > 3$, we obtain the weak convergence of the trajectories, and so find the recent  results by May and Attouch-Chbani-Peypouquet-Redont. As a new result, in the one-dimensional framework, for the critical value $\alpha = 3 $, we prove the convergence of the trajectories without any restrictive hypothesis on the convex function $\Phi $. In the second part of this paper, we study  the convergence properties of the associated forward-backward inertial algorithms.
They aim to solve structured convex minimization problems  of the form  $\min \left\lbrace \Theta:= \Phi + \Psi \right\rbrace$, with $\Phi$ smooth and $\Psi$ nonsmooth. 
The continuous dynamics serves as a guideline for this study, and is very useful for suggesting Lyapunov functions.
We obtain a similar  rate of convergence for the sequence of iterates $(x_k)$: for $\alpha \leq 3$ we have $\Theta (x_k)-\min \Theta =   \mathcal O (k^{-p})$  for all $p <\frac{2\alpha}{3}$ , and for $\alpha > 3$ \ $\Theta (x_k)-\min \Theta =   o (k^{-2})$ .
We conclude this study by    showing that the results are robust with respect to external perturbations. 

\bigskip


\vspace{0.2cm}

\end{abstract}

\maketitle

\markboth{H. ATTOUCH, Z. CHBANI, H. RIAHI}
  {SUBCRITICAL ACCELERATED GRADIENT}


\section{Introduction}
Throughout the paper, $\mathcal H$ is a real Hilbert space which is endowed with the scalar product $\langle \cdot,\cdot\rangle$, with $\|x\|^2= \langle x,x\rangle    $ for  $x\in \mathcal H$.
Let $\Phi : \mathcal H \rightarrow \mathbb R$ be a convex differentiable function. In a first part, we consider 
 the second-order differential equation 
\begin{equation}\label{edo01}
 \mbox{(AVD)}_{\alpha} \quad \quad \ddot{x}(t) + \frac{\alpha}{t} \dot{x}(t) + \nabla \Phi (x(t))  =0,
\end{equation}
where $\alpha $ is a positive parameter. Depending on  whether $\alpha$ is greater or less than the critical value $3$, we analyze the  convergence rate of the values $\Phi (x(t))-\min \Phi$, as $t\to + \infty$. The novelty of our results comes mainly from the study of the subcritical case $\alpha < 3$ whose  convergence rate was largely unknown. In a second part, we study  the corresponding inertial forward-backward algorithms in the case of structured minimization problems. The analysis of the continuous dynamics will serve  as a guideline for the study of these algorithms.
Let us first recall some historical facts, explaining the importance of these issues.

\subsection{From the heavy ball to fast optimization}
 The  heavy ball with friction system, which involves a fixed positive damping coefficient  $\gamma$
\begin{equation} \label{E:heavyball}
\ddot x(t)+\gamma\dot{x}(t)+\nabla\Phi (x(t))=0
\end{equation}
was introduced, from an optimization perspective, by Polyak \cite{Polyak2}, \cite{Polyak3}. Its convergence in the convex case was first obtained by \'Alvarez in \cite{Al}. In recent years, several studies  have been devoted to the study of the Inertial Gradient System ${\rm(IGS)_{\gamma}} $, with a time-dependent damping coefficient $\gamma(\cdot)$
\begin{equation} \label{E:IGS}
{\rm(IGS)_{\gamma}} \qquad \ddot x(t)+\gamma(t) \dot{x}(t)+\nabla\Phi (x(t))=0.
\end{equation}
A particularly interesting situation concerns the case $\gamma(t) \to 0$ of a vanishing damping coefficient. As pointed out by Su-Boyd-Cand\`es in \cite{SBC}, the $\rm{(IGS)_{\gamma} }$ system with $\gamma (t) = \frac{\alpha}{t}$, that's ${\mbox{\rm(AVD)}}_{\alpha}  $ given in (\ref{edo01}),
can be seen as a continuous version of the fast gradient
method of Nesterov (see \cite{Nest1,Nest2, Nest4}). Its adaptation to the case of structured "smooth + nonsmooth" convex optimization  gives the Fast Iterative Shrinkage-Thresholding Algorithm (FISTA) of Beck-Teboulle \cite{BT}. When $\alpha \geq 3$, the rate of convergence of these methods is $\Phi(x_k)-\min_\mathcal H \Phi= \mathcal O \left(\frac{1}{k^2}\right)$, where $k$ is the number of iterations.
Convergence of the trajectories generated by (\ref{edo01}), and of the sequences generated by Nesterov's method, has been an elusive question for decades. However, when considering (\ref{edo01}) with $\alpha >3$, it was shown by Attouch-Chbani-Peypouquet-Redont \cite{ACPR} and  May \cite{May},  that each trajectory converges weakly to an optimal solution, with the improved rate of convergence $\Phi(x(t))-\min_\mathcal H \Phi = o (\frac{1}{t^2})$. Corresponding results for the algorithmic case have been obtained by Chambolle-Dossal \cite{CD} and Attouch-Peypouquet \cite{AP}.
  The case of a general time-dependent damping coefficient $\gamma(\cdot)$ has been recently considered by Attouch-Cabot in \cite{AC1,AC2}. The latter includes the corresponding forward-backward algorithms, and unifies previous results. These results are of great importance because they are in some sense optimal: it is well known that, for first-order methods, the rate of convergence $\frac{1}{k^2}$ is the best one can expect in the worst case.
 
 \subsection{The subcritical case $\alpha \leq 3$}
 
Whereas   $\alpha >3$ has been studied in depth, the case $\alpha< 3$ has remained largely unknown. 
Our main contribution is to show that in 
the subcritical case $\alpha \leq 3$, we  still have a property of rapid convergence for the values, that now depends on $\alpha$, namely
$$\Phi (x(t))-\min_{\mathcal H} \Phi =   \mathcal O \left(\displaystyle{\frac{1}{t^{\frac{2\alpha}{3}}}}\right).$$
Of course, in the case $\alpha =3$, from this formula, we recover the well-known convergence rate  $\mathcal O(\frac{1}{t^2})$ of the  accelerated gradient method of Nesterov.
This shows a continuous variation of the  convergence rate for  the values, when $\alpha$ varies along the positive real line. The following table gives a  synthetic view, where $\alpha $ is the coefficient of the damping parameter $  \frac{\alpha}{t}$
entering the definition of  ${\mbox{\rm(AVD)}}_{\alpha}$  (defined in (\ref{edo01})).

\bigskip

\begin{large}
\begin{center}
\begin{tabular}{|c||c|c|c|}
     \hline
&&&\\     
$\alpha$ &  \qquad  $\alpha<3 \qquad$ & \qquad  $\alpha =3  \qquad $ &  \qquad  $\alpha>3  \qquad$ \\ 
&&&\\
\hline
&&&\\  
$\Phi (x(t))-\min_{\mathcal H} \Phi$ & $\mathcal O \left(\frac{1}{t^{\frac{2\alpha}{3}}}\right)$ & $\mathcal O \left(\frac{1}{t^2}\right)$ & $o \left(\frac{1}{t^2}\right)$  \\ 
 &&&\\  
     \hline
  \end{tabular}
 \end{center}
\end{large}

\bigskip

\subsection{From continuous dynamics to algorithms}
For applications, and in order to develop fast numerical splitting methods, it is important to consider "\textit{smooth + nonsmooth}" structured  optimization problems of the form
\begin{equation}\label{algo1-0}
\min \left\lbrace   \Phi (x) + \Psi (x) : \ x \in \mathcal H   \right\rbrace   
\end{equation}
where  $\Phi: \mathcal H \to  \mathbb R$ is a continuously differentiable convex function whose gradient is Lipschitz continuous, and $\Psi: \mathcal H \to  \mathbb R \cup \lbrace   + \infty  \rbrace $ is a proper lower-semicontinuous convex function. We set $\Theta:=\Phi +\Psi$, which is the convex lower-semicontinuous  function to minimize.
To handle such non-smooth optimization problem, we are naturally led to extend $\mbox{(AVD)}_{\alpha}$, and to consider the differential inclusion
\begin{equation}\label{diff-incl-0}
 \mbox{(AVD)}_{\alpha} \quad \quad \ddot{x}(t) + \frac{\alpha}{t} \dot{x}(t) +\nabla \Phi (x(t)) + \partial \Psi (x(t))  \ni 0,
\end{equation}
where $\partial \Psi$ is the subdifferential of $\Psi$. The time discretization of  
this system, implicit with respect to the nonsmooth operator $\partial \Psi$, and explicit with respect to the smooth operator $\nabla \Phi$, gives  
 the Inertial Forward-Backward algorithm
$$\mbox{\rm(IFB)}_{\alpha} \quad\left\{ 
\begin{array}{l}
y_k=x_k+(1 - \frac{\alpha}{k})(x_k-x_{k-1})\\
\rule{0pt}{15pt}
x_{k+1}=\prox_{s\Psi}(y_k-s\nabla \Phi(y_k)).
\end{array}
\right.$$
The study of the continuous dynamics will serve us as a guideline to analyze  this algorithm. Indeed, we will obtain convergence results for $\mbox{\rm(IFB)}_{\alpha}$ which are very similar those of the continuous dynamics $\mbox{(AVD)}_{\alpha}$, and which are summarized in the following table.

\bigskip

\begin{large}
\begin{center}
\begin{tabular}{|c||c|c|c|}
     \hline
&&&\\     
$\alpha$ &  \qquad  $\alpha<3 \qquad$ & \qquad  $\alpha =3  \qquad $ &  \qquad  $\alpha>3  \qquad$ \\ 
&&&\\
\hline
&&&\\  
$\Theta (x_k)-\min_{\mathcal H} \Theta$ & $\quad \mathcal O \left(\frac{1}{k^p}\right)\quad$ $\forall p < \frac{2\alpha}{3}$ & $\mathcal O \left(\frac{1}{k^2}\right)$ & $o \left(\frac{1}{k^2}\right)$  \\ 
 &&&\\  
     \hline
  \end{tabular}
 \end{center}
\end{large}

\medskip

\noindent When $\alpha <3$, in order not to have too complicated proof, we have only proved $\Theta (x_k)-\min_{\mathcal H} \Theta = \mathcal O \left(\frac{1}{k^p}\right)$ for all $p < \frac{2\alpha}{3}$. Given the results obtained in the continuous case, we can reasonably conjecture that
$\Theta (x_k)-\min_{\mathcal H} \Theta = \mathcal O \left(\frac{1}{k^{\frac{2\alpha}{3}}}\right)$.

\subsection{Organization of the paper}
In section \ref{S:continuous}, based on  new Lyapunov functions, we analyze in a unifying way the convergence rate of the values for the trajectories of $\mbox{(AVD)}_{\alpha}$. We also study the decay property of the velocity. The main novelty concerns the subcritical case $\alpha \leq 3$. In section \ref{S:conv trajectories}, we study the convergence of the trajectories. In the one-dimensional framework, we prove the convergence of trajectories in the critical case $\alpha =3$. We  also analyze the   convergence rate in   the case of a strong minimum. Then, in section \ref{algo},  based on the results for the continuous dynamics, we study the convergence  of the corresponding inertial forward-backward algorithms, and obtain very similar results.
Finally, in section \ref{S:pert}, we complete this study by   showing that the results are robust with respect to external perturbations.

\section{Convergence rates of the continuous dynamics}
\label{S:continuous}

Throughout the paper (unless otherwise stated), we assume that $\Phi : \mathcal H \rightarrow \mathbb R$ is convex and differentiable, its gradient   $\nabla \Phi$ is Lipschitz continuous on  bounded sets, and  $S := \mbox{argmin} \Phi \neq \emptyset$.

\medskip

We take for granted the existence and uniqueness of a global solution to the Cauchy problem associated with 
$ \mbox{(AVD)}_{\alpha} $ (defined in \eqref{edo01}). 
 We point out that, given $t_0>0$, for any $x_0 \in \mathcal H,  \ v_0 \in \mathcal H $, the existence of a unique global solution on $[t_0,+\infty[$ for the Cauchy problem with initial condition $x(t_0)=x_0$ and $\dot x(t_0)=v_0$ is guaranteed by the above hypothesis, see \cite[Proposition 6.2.1] {Haraux}.
Starting from $t_0 >0$ comes from the singularity of the damping coefficient $\gamma(t) = \frac{\alpha}{t}$ at zero. Indeed, since we are only concerned about the asymptotic behaviour of the trajectories, we do not really care about the origin of time. If one insists starting from $t_0 =0$, then all the results remain valid  taking $\gamma(t) = \frac{\alpha}{t+1}$.

\medskip

Our main results concerning the continuous system are given in the following theorem.

\begin{Theorem}\label{fast} 
Let $\Phi : \mathcal H \rightarrow \mathbb R$ be a convex continuously differentiable function 
such that  $\argmin \Phi$ is nonempty.
Let $x:[ t_0  ;+\infty [ \rightarrow\mathcal{H} $ be a  classical  global solution  of $ \mbox{(AVD)}_{\alpha} $.

\smallskip

i) Suppose $\alpha > 3$. Then, we have the following result of rapid convergence of the values
$$
 \Phi (x(t))-\min_{\mathcal H} \Phi =   \mathcal O (\displaystyle{\frac{1}{t^2} })  .
$$

\smallskip

ii) Suppose $0 < \alpha \leq 3$. Then the following  rate of convergence of the values is verified
$$
 \Phi (x(t))-\min_{\mathcal H} \Phi =   \mathcal O (\displaystyle{\frac{1}{t^{\frac{2\alpha}{3}}}}).
$$
Precisely
 $$\Phi (x(t))-\min_{\mathcal H} \Phi \leq \frac{C}{\displaystyle{t^{\frac{2\alpha}{3}}}}, \ \mbox{ with } \
C = t_0^{\frac{2\alpha}{3}}\left( \Phi (x(t_0))-\min_{\mathcal H} \Phi + \Vert \dot{x}(t_0)\Vert^{2}  \right) + \frac{\alpha (\alpha + 1)}{3}\mbox{\rm dist}^2 (x(t_0), \argmin \Phi).
$$
\end{Theorem}
\begin{proof}
Fix  $z\in\argmin \Phi$, and consider the energy function $ \mathcal{E}_{\lambda, \xi}^{p}: [t_0, +\infty[ \to \mathbb R^+$
\begin{equation} \label{E:basic-Lyap}
\mathcal{E}_{\lambda, \xi}^{p}(t)=t^{2p}\left[ \Phi(x(t))-\min_{\mathcal H} \Phi\right] + \frac{1}{2}\Vert \lambda(t)(x(t)-z)+t^p \dot{x}(t)\Vert^{2}+\frac{\xi(t)}{2}\Vert x(t)-z\Vert^2 ,
\end{equation}
that will serve as a Lyapunov function. In the definition of $\mathcal{E}_{\lambda, \xi}^{p}$, $p$ is a  positive real number, and $\lambda(\cdot)$,  $\xi(\cdot )$ are  positive functions. They will be appropriately chosen during the proof, so as to make $\mathcal{E}_{\lambda, \xi}^{p}(\cdot)$ a nonincreasing function.  
Let us derivate $\mathcal{E}_{\lambda}^{p}(\cdot)$ by means of the classical derivation chain rule.  By relying on  the convexity of $\Phi(\cdot)$, and using (\ref{edo01}), we obtain, after simplification,
\begin{equation} \label{E:basic-ineq}
\begin{array}{lll}
\frac{d}{dt}\mathcal{E}_{\lambda, \xi}^{p}(t)&\leq &t^p \left[ 2pt^{p-1}-\lambda(t))\right] \left(  \Phi(x(t))-\min_{\mathcal H} \Phi\right)  \\
&  +& \left[ \xi (t)+t^{p}\dot{\lambda}(t)-(\alpha -p)t^{p-1}\lambda(t)+\lambda(t)^{2}\right]\langle x(t)-z, \dot{x}(t)\rangle  \\
&-& t^{p}\left[ (\alpha -p)t^{p-1}-\lambda(t)\right] \Vert \dot{x}(t)\Vert^{2} + \left[ \lambda(t)\dot{\lambda}(t)+\frac{\dot{\xi}(t)}{2}\right] \Vert x(t)-z\Vert^{2}. 
\end{array}
\end{equation}
Let us successively examine the differents terms entering the second member of  \eqref{E:basic-ineq}:

\medskip

Let us  make the  first two terms equal to zero by taking respectively

\medskip

 $ \quad (\textbf{H}_1)$: \  $\lambda(t)=2pt^{p-1}$ 

\medskip

\indent $\quad (\textbf{H}_2)$: \ $\xi (t)+t^{p}\dot{\lambda}(t)-(\alpha -p)t^{p-1}\lambda(t)+\lambda(t)^{2} =0$.

\medskip

From $(\textbf{H}_1)$ this implies 

\smallskip

$\quad \xi(t)=2(\alpha -4p+1)pt^{2(p-1)}.$

\medskip

 To ensure that $\mathcal{E}_{\lambda, \xi}^{p}$ is nonnegative, we must impose $\xi(\cdot)\geq 0$. This is equivalent to assuming that

\smallskip

$\quad (\textbf{H}_3 )$ \ $\alpha\geq 4p-1$.

\smallskip

 Recall that we want $\mathcal{E}_{\lambda, \xi}^{p}(\cdot)$ to be  a nonincreasing function. Hence, we impose on the third term of the second member of \eqref{E:basic-ineq} to satisfy $(\alpha -p)t^{p-1}-\lambda(t)\geq 0$. 
This is equivalent to assuming that

\smallskip

$ \quad (\textbf{H}_4 )$ \ $ \alpha \geq 3p$.

\smallskip

With the above choice of the parameters $\lambda(\cdot), \xi(\cdot), p$, we obtain
\begin{equation}\label{beta}
\frac{d}{dt}\mathcal{E}_{\lambda, \xi}^{p}(t)\leq \beta(t)\Vert x(t)-z\Vert^{2},
\end{equation}
where $\beta(t)= \lambda(t)\dot{\lambda}(t)+\dfrac{\dot{\xi}(t)}{2}$. A straightforward calculation gives
$\beta(t)=-2p(1-p)(\alpha-2p+1)t^{2p-3}$.  We want $\beta(\cdot)$ to be less than or equal to zero. Since  
 $\alpha\geq2p-1$ is implied
  by $(\textbf{H}_4 )$ (or $(\textbf{H}_3 )$), we  impose  the  supplementary condition:

\medskip

\indent  $\quad (\textbf{H}_5 )$ \ $1\geq p$.

\smallskip

Let us put together the  results above.
By taking $p=\min (1,\frac{\alpha}{3}, \frac{\alpha+1}{4})$, and $\lambda(\cdot)$, $\xi(\cdot)$ given respectively by  $(\textbf{H}_1)$ and $(\textbf{H}_2)$, all the conditions $(\textbf{H}_1 )$ to $(\textbf{H}_5 )$ are verified.
As a consequence,  $\frac{d}{dt}\mathcal{E}_{\lambda,\xi}^{p}(t)\leq 0$,  and   $\mathcal{E}_{\lambda,\xi}^{p}$ is nonincreasing. Coming back to the definition of $\mathcal{E}_{\lambda}^{p}(\cdot)$, taking account that $\xi (\cdot)$ is nonegative, we deduce that, for all $t\geq t_0$
\begin{equation}\label{E:est0}
 \Phi (x(t))-\min_{\mathcal H} \Phi \leq \frac{\mathcal{E}_{\lambda,\xi}^{p}(t_0)}{t^{2p}}  .
\end{equation}
A closer look at the formula $p=\min (1,\frac{\alpha}{3}, \frac{\alpha+1}{4})$ shows that this expression can be simplified as 
$p= \min (1,\frac{\alpha}{3})$. This leads us to distinguish the two cases:

\smallskip

$\bullet$ $\alpha \geq 3.$ Then $p=\min (1,\frac{\alpha}{3})=1$, which, from \eqref{E:est0}, gives  item $(i)$ of Theorem \ref{fast}.

\medskip

$\bullet$ $\alpha \leq 3.$ Then $p=\min (1,\frac{\alpha}{3})= \frac{\alpha}{3}$, which, from \eqref{E:est0}, gives item $(ii)$.
Let us  make precise the value of the constant that appears in the corresponding  estimation
\begin{equation}\label{E:est1}
 \Phi (x(t))-\min_{\mathcal H} \Phi \leq \frac{C}{t^{\frac{2\alpha}{3}}} .
\end{equation}
 By definition \eqref{E:basic-Lyap} of  $\mathcal{E}_{\lambda,\xi}^{p}(t_0)$ and  elementary computation we obtain as value of the constant $C$
 \begin{equation}\label{E:est2}
 C = t_0^{\frac{2\alpha}{3}}\left( \Phi (x(t_0))-\min_{\mathcal H} \Phi + \Vert \dot{x}(t_0)\Vert^{2}  \right) + \frac{\alpha (\alpha + 1)}{3}\mbox{\rm dist}^2 (x(t_0), \argmin \Phi),
\end{equation}
which completes the proof of Theorem \ref{fast}.
\end{proof}
\begin{remark}
The proof of the Theorem \ref{fast} is based on the use of the Lyapunov function $\mathcal{E}_{\lambda, \xi}^{p}$, which, with the specific choices of the parameters 
$p= \min (1,\frac{\alpha}{3})$, $\lambda(t)=2pt^{p-1}$ , $\xi(t)=2(\alpha -4p+1)pt^{2(p-1)}$,  is written as follows
\begin{align*} 
\mathcal{E}_{\lambda, \xi}^p(t)=t^{2\min (1,\frac{\alpha}{3})}\left[ \Phi(x(t))-\min_{\mathcal H} \Phi\right] &+ \frac{1}{2}\Vert 2\min (1,\frac{\alpha}{3})t^{\min (1,\frac{\alpha}{3})-1}(x(t)-z)+t^{\min (1,\frac{\alpha}{3})} \dot{x}(t)\Vert^{2}\\
&+(\alpha -4\min (1,\frac{\alpha}{3})+1)\min (1,\frac{\alpha}{3})t^{2(\min (1,\frac{\alpha}{3})-1)}\Vert x(t)-z\Vert^2 .
\end{align*}
Using this Lyapunov function, we were able to provide a unified proof of the 
convergence results in the two different cases $\alpha \leq 3$, and $\alpha \geq 3$.
Obviously, the complexity of this formula explains why historically these two cases have been considered independently.
In each of these cases, the formula simplifies significantly:

\smallskip

$i)$ In the case $\alpha \geq 3$, we have
$p=1$, $\lambda (t)= 2$, $\xi(t) = 2(\alpha -3)$, which gives
$$
\mathcal{E}_{\lambda, \xi}^1(t)= t^2 \left[ \Phi(x(t))-\min_{\mathcal H} \Phi\right] + \frac{1}{2}\Vert 2(x(t)-z)+t \dot{x}(t)\Vert^{2}+ (\alpha -3)\Vert x(t)-z\Vert^2.
$$
This is precisely the Lyapunov function used in the proof of Theorem 2.14 in \cite{ACPR}.

\smallskip

$ii)$  In the case $\alpha \leq 3$, we have
$p=\frac{\alpha}{3}$, $\lambda(t)=\frac{2\alpha}{3}t^{\frac{\alpha}{3}-1}$ , $\xi(t)=\frac{2\alpha}{3}(\alpha -\frac{4\alpha}{3}+1)t^{2(\frac{\alpha}{3}-1)}$, which gives
$$
\mathcal{E}_{\lambda, \xi}^{\frac{\alpha}{3}}(t)= t^{\frac{2\alpha}{3}} \left[ \Phi(x(t))-\min_{\mathcal H} \Phi\right] + \frac{1}{2}\Vert \frac{2\alpha}{3}t^{\frac{\alpha}{3}-1}(x(t)-z)+ t^{\frac{\alpha}{3}} \dot{x}(t)\Vert^{2}+ \frac{\alpha}{3}(\alpha -\frac{4\alpha}{3}+1)t^{2(\frac{\alpha}{3} -1)}\Vert x(t)-z\Vert^2.
$$
Note that it is rather difficult to postulate a priori these formulas:  it was by an identification technique that we could find the appropriate parameters.
\end{remark}

\begin{remark}
Let us give  a synthetic view of the rate of convergence of the values for the trajectories of $ \mbox{(AVD)}_{\alpha} $. Given $x(\cdot)$ a global solution trajectory  we have
$$
\Phi (x(t))-\min_{\mathcal H} \Phi = \mathcal O \left(\frac{1}{t^{p(\alpha)}}\right) \quad \mbox{with}\quad p(\alpha) = \min\left( \frac{2\alpha}{3} ,2\right).
$$
 
This formula and the figure below show the two distinct regimes. On the interval $]0, 3]$ the exponent $p(\alpha)$ increases linearly from zero to $2$. Then, after $\alpha =3$ there is a plateau,  the exponent $p(\alpha)$ remains constant equal to $2$. Of course, this is consistent with the Nesterov complexity bound \cite{Nest1}, \cite{Nest2}, which tells us that for first-order methods, the rate of convergence $\frac{1}{k^2}$ is the best one can expect in the worst case. See Drori-Teboulle \cite{DT} and  Kim-Fessler \cite{Kim-F} for a recent account on first-order  methods that achieve best performance.
\end{remark}
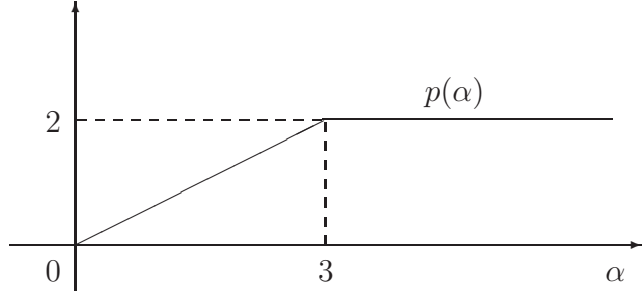
\begin{Large}
\setlength{\unitlength}{12cm}
\begin{picture}(0.4,0.6)(-0.4, 0.13)
\put(0.65,0.26){$\alpha$}
\put(0.03,0.26){0}
\put(0.332,0.26){3}
\put(0.03,0.425){2}
\put(0.45,0.46){$p(\alpha)$}
\multiput(0.065,0.439)(0.02,0){14}
{\line(1,0){0.01}}
\multiput(0.34,0.3)(0,0.021){7}
{\line(0,1){0.01}}
\put(-0.01,0.3){\vector(1,0){0.70}}
\put(0.063,0.25){\vector(0,1){0.32}}

\put(0.337,0.44){\line(1,0){0.32}}
\put(0.063,0.3){\line(2,1){0.275}}
\end{picture} 
\end{Large}

Let us complement the pointwise estimates concerning 
$\Phi(x(t))-\min \Phi$ given in Theorem \ref{fast} by an integral estimate.

\begin{Theorem}\label{integral-decay}
Let $\Phi : \mathcal H \rightarrow \mathbb R$ be a convex continuously differentiable function 
such that  $\argmin \Phi$ is nonempty.
Let $x:[ t_0  ;+\infty [ \rightarrow\mathcal{H} $ be a  classical  global solution  of $ \mbox{(AVD)}_{\alpha} $. Then, for any $p$ satisfying $p\leq 1$ and $p<\frac{\alpha}{3}$, the following inequality is satisfied
$$\int_{t_0 }^{+\infty}t^{2p-1}(\Phi(x(t))-\min \Phi )dt\leq \frac{\mathcal{E}_{\lambda, \xi}^{p}(t_0 )}{(\alpha-3p)},$$
where $\lambda (t)= (\alpha -p)t^{p-1}$ and $\xi (t)=(1-p)(\alpha -p)t^{2(p-1)}$.
Thus, 

\medskip

i) If $\alpha >3$, then,   $\displaystyle{\int_{t_0 }^{+\infty}t(\Phi(x(t))-\min \Phi )dt < +\infty}$.

\medskip

ii) If $\alpha \leq 3$, then for any $p < \frac{2\alpha}{3} -1$, \  $\displaystyle{\int_{t_0 }^{+\infty}t^p(\Phi(x(t))-\min \Phi )dt < +\infty.}$

\smallskip

In particular for $\alpha = 3$, we have $\displaystyle{\int_{t_0 }^{+\infty}t^p(\Phi(x(t))-\min \Phi )dt < +\infty}$ for all $p <1 $.
\end{Theorem}
\begin{proof}
Let us slightly modify the  choices concerning the assumptions $(\textbf{H}_1 )$ to $(\textbf{H}_5 )$ that have been  made in the proof of Theorem \ref{fast}. Still we want to maintain the property $\frac{d}{dt}\mathcal{E}_{\lambda,\xi}^{p}(t)\leq 0$.
Let us assume 
\begin{enumerate}
\item $ \lambda(t)\geq 2pt^{p-1}$

\medskip

\item $ (\alpha -p)t^{p-1}-\lambda(t) = 0$.
\end{enumerate}
Clearly, these two conditions are equivalent to  
$$
 \lambda(t) = (\alpha -p)t^{p-1} \ \mbox{ with } \ \alpha \geq 3p .
 $$
On the other hand, condition $(\textbf{H}_2 )$ becomes 
$$\xi(t)=-t^p \dot{\lambda}(t)=(1-p)(\alpha -p)t^{2(p-1)}$$
 which gives
$$\beta(t) =\lambda(t)\dot{\lambda}(t)+\dfrac{\dot{\xi}(t)}{2}=-(1-p)(\alpha -p)(\alpha-2p+1)t^{2p-3}.$$
Altogether taking $p\leq \min(1,\frac{\alpha}{3})$ ensures that  $\xi(\cdot)\geq 0$, $\beta(\cdot )\leq 0$, and
$$\frac{d}{dt}\mathcal{E}_{\lambda, \xi}^{p}(t)+ \left[ (\alpha -3p)t^{2p-1}\right] \left(  \Phi(x(t))-\min \Phi\right) \leq 0.$$
Then we integrate, and  use  $\mathcal{E}_{\lambda, \xi}^{p}(t)\geq 0$, $\alpha -3p>0$ to conclude.
\end{proof}
\begin{remark}
In the case of the Nesterov accelerated gradient method, that is $\alpha =3$, the estimation 
$$\forall \epsilon >0 \qquad \int_{t_0 }^{+\infty}t^{1-\epsilon}(\Phi(x(t))-\min \Phi )dt < +\infty$$
 is new, to our knowledge. This formula is known to be true with $\epsilon =0$ in the case $\alpha >3$, see for example \cite{ACPR}, in which case it is the key ingredient of the proof of weak convergence of the trajectories of $ \mbox{(AVD)}_{\alpha} $.
\end{remark}

Let us supplement the asymptotic analysis by examining the rate of decay of the speed. First, we establish integral estimates, and then we build on these results  to obtain sharp pointwise estimates.

\begin{Theorem}\label{speed-decay} (\textit{integral estimates of the speed}) \
 Let $x:[ t_0  ;+\infty [ \rightarrow\mathcal{H} $ be a solution of $ \mbox{(AVD)}_{\alpha} $.
 The following integral estimates are satisfied:
 \begin{enumerate} 
 \item If $\alpha > 3$, then
 $$
\int_{t_0 }^{+\infty} t \Vert\dot{x}(t)\Vert^2 dt < + \infty. 
$$
\item If $\alpha \leq 3$, then
$$
\int_{t_0 }^{+\infty} t^p \Vert\dot{x}(t)\Vert^2 dt < + \infty  \quad \mbox{for all } \  p < \alpha -2. 
$$
In particular for $\alpha =3$ we have 
$$
\int_{t_0 }^{+\infty} t^{1-\epsilon} \Vert\dot{x}(t)\Vert^2 dt < + \infty  \quad \mbox{for all } \  \epsilon >0 . 
$$
 \end{enumerate}
\end{Theorem}

\begin{proof}
(1) Let us keep the same choice of the parameters as that done in the proof of  Theorem \ref{fast}, that is     $\lambda(t)=2pt^{p-1}$,
 $ \xi(t)=2(\alpha -4p+1)pt^{2(p-1)}$, and  $p=\min (1,\frac{\alpha}{3})$.
As a consequence $\mathcal{E}_{\lambda, \xi}^{p}(\cdot)$ is nonincreasing, and from \eqref{E:basic-ineq} we have
$$
\frac{d}{dt}\mathcal{E}_{\lambda, \xi}^{p}(t) + t^{p}\left[ (\alpha -p)t^{p-1}-\lambda(t)\right] \Vert \dot{x}(t)\Vert^{2} \leq 0.
$$
  for all $t\geq t_0$. 
Equivalently
\begin{equation} \label{energy10}
\frac{d}{dt}\mathcal{E}_{\lambda, \xi}^{p}(t)+  (\alpha -3p)t^{2p-1}\Vert\dot{x}(t)\Vert^2 \leq 0.
\end{equation}
This inequality gives us information about the rate of decay of the energy function $\mathcal{E}_{\lambda, \xi}^{p}$ only when $\alpha > 3p= \min (\alpha, 3)$.
Obviously, this later condition is equivalent to $\alpha >3$, which gives $p=1$, and $2p-1=1$. By integrating (\ref{energy10}) on $[t_0 , t]$ we obtain
$$\mathcal{E}_{\lambda, \xi}^{p}(t)+ (\alpha-3)\int_{t_0 }^t s \Vert\dot{x}(s)\Vert^2 ds \leq \mathcal{E}_{\lambda, \xi}^{p}(t_0).$$
Thus, in the case $\alpha >3$ we recover the well-known estimate (see for example \cite{ACPR}, \cite{AP})
$$
\int_{t_0 }^{+\infty} t \Vert\dot{x}(t)\Vert^2 dt < + \infty. 
$$
$(2)$ Let us now assume $\alpha \leq 3$. 
Let us slightly modify the  choices concerning the hypotheses $(\textbf{H}_1 )$ to $(\textbf{H}_5 )$ which have been  made in the proof of the Theorem \ref{fast}. We always want to maintain the property $\frac{d}{dt}\mathcal{E}_{\lambda,\xi}^{p}(t)\leq 0$, and $\xi(\cdot) \geq 0$.
Let us modify the condition $(\textbf{H}_1)$ as follows: take
$
\lambda (t) = \mu t^{p-1},
$
with $\mu >0$ and

\smallskip

$ (i) \ \mu \geq 2p ,$

\smallskip

\noindent so as to have $t^p \left[ 2pt^{p-1}-\lambda(t))\right]\leq 0$.  Then, condition $(H_2)$ yields
$
\xi(t)= \mu t^{2p-2}(\alpha +1-2p -\mu).
$
Having $\xi(\cdot) \geq 0$ gives 

\medskip

$ (ii) \ \mu \leq \alpha +1-2p  .$

\smallskip

\noindent We  now have $\beta(t)= \lambda(t)\dot{\lambda}(t)+\dfrac{\dot{\xi}(t)}{2}=-\mu(1-p)(\alpha+1-2p)t^{2p-3}$. 
 We want $\beta(\cdot)$ to be less or equal than zero. By $(ii)$ we already have
 $ \alpha +1-2p \geq  \mu >0$. Thus 
  we only need to impose  the  supplementary condition on $p$:

\medskip

 $\quad (iii)$ \ $ p \leq 1$.
 
 \smallskip

\noindent With the above choices $(i),  (ii), (iii)$
inequality  \eqref{E:basic-ineq} becomes
\begin{equation*}
\frac{d}{dt}\mathcal{E}_{\lambda, \xi}^{p}(t) 
+  t^{p}\left[ (\alpha -p)t^{p-1}-\lambda(t)\right] \Vert \dot{x}(t)\Vert^{2}  \leq 0.
\end{equation*}
Since $\lambda (t) = \mu t^{p-1},$ we have equivalently
\begin{equation} \label{E:basic-ineq-c}
\frac{d}{dt}\mathcal{E}_{\lambda, \xi}^{p}(t) 
+  \left[ \alpha -p -\mu \right] t^{2p-1}\Vert \dot{x}(t)\Vert^{2}  \leq 0.
\end{equation}
Hence, assuming

\medskip

 $\quad (iv)$ \ $ \alpha -p -\mu >0$,
 
 \smallskip
 
\noindent by integration of \eqref{E:basic-ineq-c} we obtain
 \begin{equation} \label{E:basic-ineq-d}
\int_{t_0 }^{+\infty} t^{2p-1} \Vert\dot{x}(t)\Vert^2 dt < + \infty. 
\end{equation}
Let us now examine for which values of $p$ the  conditions $(i)$ to $(iv)$ are compatible.
Firstly, $(i)$ and $(ii)$ give $2p \leq \mu  \leq \alpha +1 -2p$. To ensure that these inequalities are compatible leads us to make the hypothesis

\smallskip

 $\quad (v)$ \ $ p \leq \frac{\alpha +1}{4} $.
 
 \smallskip
 
\noindent Take as a value of $\mu$ the middle point of the non void interval $[2p , \alpha +1 -2p  ]$, that is
$$
\mu= \frac{\alpha +1}{2}.
$$
Then the condition $ \alpha -p -\mu >0$ become
\smallskip

 $\quad (vi)$ \ $ p < \frac{\alpha -1}{2} $.
 
 \smallskip
 
\noindent Putting together the conditions $(iii)$, $(v)$, and $(vi)$, we finally obtain the  condition on $p$ 
$$
0 < p < \min\left(1, \frac{\alpha +1}{4}, \frac{\alpha -1}{2}\right).
$$
A close look at this formula shows that for $1 < \alpha \leq 3$ it is equivalent to 
$
0< p < \frac{\alpha -1}{2}.
$
This gives $2p-1 < \alpha -2$, which combined with 
\eqref{E:basic-ineq-d} gives the result.
\end{proof}

\begin{Theorem}\label{speed-decay-pointwise} (\textit{pointwise estimates of the speed}) \
 Let $x:[ t_0  ;+\infty [ \rightarrow\mathcal{H} $ be a solution of $ \mbox{(AVD)}_{\alpha} $.
 \begin{enumerate}
 \item If $\alpha \geq 3$, the trajectory satisfies
 $$
 \sup_{t\geq t_0} \Vert x(t)\Vert < + \infty
\mbox{\quad and \quad} 
\Vert \dot{x}(t)\Vert = \mathcal O (\frac{1}{t}).$$
 \item If $\alpha > 3$, we have 
 $$\Vert \dot{x}(t)\Vert = o (\frac{1}{t}).$$
 \item If $1 \leq \alpha \leq 3$, we have for all $p < \frac{\alpha -1}{2}$
 $$\Vert \dot{x}(t)\Vert = \mathcal O \left(\frac{1}{t^p}\right).
 $$
 \end{enumerate}
\end{Theorem}

\begin{proof} 
 
$(1)$ Let us keep the same choice of the different parameters as that done in the proof of  Theorem \ref{fast}, that is     $\lambda(t)=2pt^{p-1}$,
 $ \xi(t)=2(\alpha -4p+1)pt^{2(p-1)}$, and  $p=\min (1,\frac{\alpha}{3})$.
As a consequence $\mathcal{E}_{\lambda, \xi}^{p}(\cdot)$ is nonincreasing, and since
$\xi(\cdot)$ is nonnegative, we deduce that, for all $t\geq t_0$
\begin{equation}\label{pointwise-est1}
\mathcal{E}_{\lambda, \xi}^{p}(t_0)  \geq \mathcal{E}_{\lambda, \xi}^{p}(t)\geq \frac{1}{2}\Vert 2pt^{p-1}(x(t)-z)+t^p \dot{x}(t)\Vert^2 .
\end{equation}
After developping the above quadratic term, and neglecting the nonnegative term $t^{2p}\Vert \dot{x}(t)\Vert^2$ we obtain
$$
2p^2t^{2p-2}\Vert x(t)-z\Vert^2  + 2pt^{2p-1}\left\langle x(t)-z,  \dot{x}(t) \right\rangle \leq \mathcal{E}_{\lambda, \xi}^{p}(t_0).
$$
Setting $h(t):= \Vert x(t)-z\Vert^2   $, we have equivalently 
$$
2pt^{2p-2}h(t)  + t^{2p-1}\dot{h}(t) \leq \frac{1}{p}\mathcal{E}_{\lambda, \xi}^{p}(t_0).
$$
Then note that
$$
\frac{d}{dt}t^{2p-1}h(t)= (2p-1)t^{2p-2}h(t) + t^{2p-1}\dot{h}(t) \leq 2pt^{2p-2}h(t)  + t^{2p-1}\dot{h}(t).
$$
Combining the two above inequalities we obtain
$$
\frac{d}{dt}t^{2p-1}h(t)\leq \frac{1}{p}\mathcal{E}_{\lambda, \xi}^{p}(t_0),
$$
which by integration gives
$$
\Vert x(t)-z\Vert^2   \leq \frac{t_0^{2p-1}}{t^{2p-1}}\Vert x(t_0)-z\Vert^2  +
\frac{t-t_0}{pt^{2p-1}}\mathcal{E}_{\lambda, \xi}^{p}(t_0).
$$
Hence for $2p-1 \geq 1$, i.e., $p \geq 1$ the trajectory remains bounded. Since $p=\min (1,\frac{\alpha}{3})$, this is equivalent to suppose $\alpha \geq 3$, and $p=1$.
Returning to (\ref{pointwise-est1}), and by the triangle inequality we immediately infer 
$$t\Vert\dot{x}(t)\Vert\leq \sqrt{2\mathcal{E}_{\lambda, \xi}^{1}(t_0 )}+2\Vert x(t)-z\Vert.$$
Hence

 $$\Vert \dot{x}(t)\Vert\leq \frac{\sqrt{2\mathcal{E}_{\lambda, \xi}^{1}(t_0 )}}{t }+\frac{2}{t}\sup_{t\geq t_0 }\Vert x(t)-z\Vert 
 $$
 with $\lambda=2$,
 $ \xi=2(\alpha -3)$. As a result, $\Vert \dot{x}(t)\Vert = \mathcal O (\frac{1}{t})$, which gives us item $(1)$.

\medskip

The next part of the proof is an adaptation to our framework of the technique developed by Attouch-Peypouquet in \cite{AP}. It is based on the decay properties of the scaled energy function
$$
\Gamma (t):= t^{2p} W(t)
$$
where $W$ is the global energy defined by
$$
W(t):= \Phi (x(t))-\min_{\mathcal H} \Phi + \frac{1}{2}  \Vert \dot{x}(t)\Vert^2 .
$$
By a direct application of the derivation chain rule, and using equation $ \mbox{(AVD)}_{\alpha} $, we have 
$$
\dot{W}(t) = -\frac{\alpha}{t}\Vert \dot{x}(t)\Vert^2 .
$$
Hence
\begin{align*}
\dot{\Gamma} (t)&= 2p t^{2p-1} W(t)+ t^{2p}\dot{W}(t)
\\
&= (p-\alpha)t^{2p-1}\Vert \dot{x}(t)\Vert^2 + 2p t^{2p-1}(\Phi (x(t))-\min_{\mathcal H} \Phi). 
\end{align*}
Taking  $p \leq \alpha$, we obtain
\begin{equation}\label{E:basic-decay1}
\dot{\Gamma} (t) \leq 2p t^{2p-1}(\Phi (x(t))-\min_{\mathcal H} \Phi).
\end{equation}
Let us successively examine the case $\alpha >3$ and $\alpha \leq 3$.
\medskip

$(2)$ Suppose $\alpha >3$.  In that case we take $p=1$, and hence     
$$
\Gamma (t):= t^{2} W(t).
$$
Since $p=1 \leq 3$, the inequality \eqref{E:basic-decay1} gives
$$
\dot{\Gamma} (t) \leq 2 t(\Phi (x(t))-\min_{\mathcal H} \Phi).
$$
By Theorem \ref{integral-decay} item $i)$, we have
$$
\int_{t_0 }^{+\infty} t(\Phi (x(t))-\min_{\mathcal H} \Phi) dt < + \infty . 
$$
Hence,
$
\left[ \dot{\Gamma} \right]^+ \in L^1 (t_0, + \infty)
.$
Since $\Gamma$ is nonnegative, this implies that $\lim \Gamma (t)$ exists. On the other hand,
$$
\frac{1}{t}\Gamma (t) = t W(t)= t(\Phi (x(t))-\min_{\mathcal H} \Phi ) + \frac{t}{2}  \Vert \dot{x}(t)\Vert^2 .
$$
By Theorem \ref{speed-decay} item $(1)$, we have
$t \Vert \dot{x}(t)\Vert^2 \in
L^1 (t_0, + \infty)$. Applying again Theorem \ref{integral-decay} item $i)$ we have 
$ t(\Phi (x(t))-\min_{\mathcal H} \Phi )  \in L^1 (t_0, + \infty) $.
By combining these results, we deduce that
 $\frac{1}{t} \Gamma (t) \in L^1 (t_0, + \infty)$. 
The function $\Gamma $ verifies $\frac{1}{t}\Gamma (t) \in L^1 (t_0, + \infty)$ and $\lim \Gamma (t)$ exists.
Hence, $\lim \Gamma (t) =0$.
Equivalently $\lim t^{2}\Vert \dot{x}(t)\Vert^2  =0$,
 which gives  the claim.
 
 \medskip

3) Now suppose  $\alpha \leq 3$.  In this case, we follow an argument similar to the one above but with a parameter $ p \leq \alpha $ which will be chosen during the proof in a convenient way. Let us return to 
\eqref{E:basic-decay1}. By Theorem \ref{integral-decay}, we have 
$t^{2p-1}(\Phi (x(t))-\min_{\mathcal H} \Phi ) \in
L^1 (t_0, + \infty)
$ for $p\leq 1$ and $p<\frac{\alpha}{3}$. 
As a consequence, under these two conditions on $p$ we have
$
\left[ \dot{\Gamma} \right]^+ \in L^1 (t_0, + \infty).
$
Since $\Gamma$ is nonnegative, this implies that $\lim \Gamma (t)$ exists. On the other hand,
$$
\frac{1}{t}\Gamma (t) = t^{2p-1}(\Phi (x(t))-\min_{\mathcal H} \Phi ) + \frac{t^{2p-1}}{2}  \Vert \dot{x}(t)\Vert^2 .
$$
Applying again Theorem \ref{integral-decay}, we have 
$t^{2p-1}(\Phi (x(t))-\min_{\mathcal H} \Phi ) \in
L^1 (t_0, + \infty)
$ for $p\leq 1$ and $p<\frac{\alpha}{3}$.
Moreover, by Theorem \ref{speed-decay} item $(2)$, we have
$t^{2p-1} \Vert \dot{x}(t)\Vert^2 \in
L^1 (t_0, + \infty)
$  for $2p-1 < \alpha -2$, that is $p < \frac{\alpha -1}{2}$.
Putting all these conditions together we have obtained that  for $p < \min(1, \frac{\alpha}{3}, 
\frac{\alpha -1}{2},\alpha)$, $\lim \Gamma (t)$ exists, and $\frac{1}{t} \Gamma (t) \in L^1 (t_0, + \infty)$. Then observe that $\min(1, \frac{\alpha}{3},\frac{\alpha -1}{2}, \alpha)= \frac{\alpha -1}{2}$ when $\alpha \leq 3$.
As a consequence, when $\alpha \leq 3$ we have  $\lim \Gamma (t) =0$ for all $p < \frac{\alpha -1}{2}$.
Hence $\lim t^{2p}\Vert \dot{x}(t)\Vert^2  =0$
for all $p < \frac{\alpha -1}{2}$ which is the claim (it is equivalent to state the result with $\mathcal O$ or small $o$).
\end{proof}

As a direct consequence of $\lim \Gamma (t) = \lim t^2 W(t)=0$, in the case $\alpha >3$ we recover the following result of  Attouch-Peypouquet \cite{AP}.

\begin{corollaire} For $\alpha >3$, and for any solution trajectory $x(\cdot)$ of $ \mbox{(AVD)}_{\alpha} $ we have 
$$
\Phi (x(t))-\min_{\mathcal H} \Phi =o \left(\frac{1}{t^2}\right)
$$
\end{corollaire}

We can now complete the table giving a synthetic view of the rate of convergence for the values and the speed of the solution trajectories of $ \mbox{(AVD)}_{\alpha} $.

\bigskip

\begin{large}
\begin{center}
\begin{tabular}{|c||c|c|c|}
     \hline
&&&\\     
$\alpha $ &   $\alpha< 3$ & $\alpha =3 $  &  $\qquad \alpha > 3 \qquad $ \\
\hline
&&&\\  
$\Phi (x(t))-\min_{\mathcal H} \Phi$ & $\mathcal O \left(\frac{1}{t^{\frac{2\alpha}{3}}}\right)$ & $\mathcal O \left(\frac{1}{t^2}\right)$ & $o \left(\frac{1}{t^2}\right)$  \\ 
 &&&\\  
     \hline
     &&&\\  
$I_p:= \int_{t_0 }^{+\infty}t^p(\Phi(x(t))-\min \Phi )dt$ & $\quad I_p < +\infty \quad$  $ \forall p < \frac{2\alpha}{3} -1 \quad$ & $I_p < +\infty \quad$ \  $\forall p <1 $ & $I_1< +\infty$ \\ 
 &&&\\  
     \hline
     &&&\\
  $\|\dot{x}(t)\|$ & $\mathcal O \left(\frac{1}{t^p}\right)\quad $  $ \forall p < \frac{\alpha -1}{2}$  & $\mathcal O \left(\frac{1}{t}\right)$& $o \left(\frac{1}{t}\right)$ \\       
      &&&\\  
     \hline
     &&&\\   
$J_p:= \int_{t_0 }^{+\infty} t^p \Vert\dot{x}(t)\Vert^2 dt$ & $J_p < +\infty \quad$ \  $ \forall p < \alpha -2$ & $J_p < +\infty \quad$ \  $\forall p <1 $ & $J_1< +\infty$ \\ 
 &&&\\  
     \hline
  \end{tabular}
 \end{center}
\end{large}

\bigskip

\section{Convergence of the trajectories}\label{S:conv trajectories}

\subsection{Weak convergence for $\alpha > \textbf{3}$} \
Let us recall the convergence result in the case $\alpha >3$ obtained by 
by Attouch-Chbani-Peypouquet-Redont \cite{ACPR} and  May \cite{May} in the case $\alpha >3$. We  give a brief demonstration of it, since it can  be obtained as a direct consequence of our previous results, and this enlights the situation in the case $\alpha \leq 3$.
\begin{Theorem}\label{w-convergence}
Let $\Phi : \mathcal H \rightarrow \mathbb R$ be a convex continuously differentiable function 
such that  $\argmin \Phi$ is nonempty.
Let $x:[ t_0  ;+\infty [ \rightarrow\mathcal{H} $ be a  classical  global solution  of $ \mbox{(AVD)}_{\alpha} $ with $\alpha >3$. Then $x(t)$ converges weakly, as $t\to +\infty$ to a point in $\argmin \Phi$.
\end{Theorem}

\begin{proof} The proof is  based on the  Opial's lemma \ref{Opial}.  By elementary calculus, convexity of $\Phi$, and equation $ \mbox{(AVD)}_{\alpha} $, one can first establish that for any $z \in \argmin \Phi$, the function $h_z (t):= \frac{1}{2}\|x(t)-z  \|^2$ satisfies
\begin{equation}\label{E:weak-conv2}
t\ddot{h}_z(t) + \alpha\dot{h}_z(t)\leq t \|  \dot{x} (t)\|^2 .
\end{equation}
By  integrating the differential inequality \eqref{E:weak-conv2},
and using the estimate $
\int_{t_0 }^{+\infty} t \Vert\dot{x}(t)\Vert^2 dt < + \infty$ given in Theorem \ref{speed-decay} in the case $\alpha>3$, 
we infer
\begin{equation}\label{E:weak-conv3}
\left[ \dot{h}_z\right]^+ \in L^1 (t_0, + \infty).
\end{equation}   
Since $h_z$ is nonegative, this implies the convergence of $h_z$. The second item of the Opial's lemma is a direct consequence of the minimizing property of the trajectory, and of the convexity of $\Phi$.
\end{proof}
\subsection{Critical case:}\label{crit} $\alpha \textbf{=}\textbf{3}.$ \
As we have already pointed out, the convergence of the trajectories of $ \mbox{(AVD)}_{\alpha} $ in the case $\alpha =3$ remains a widely open question.
In the one-dimensional setting, and in the case $\alpha =3$,
as a main result, in this section, we prove the convergence  of  $ \mbox{(AVD)}_{\alpha} $ trajectories, without any restrictive assumption on the convex potential $\Phi$. This comes as a generalization of  Cabot-Engler-Gadat \cite[Theorem 3.1]{CEG2}, who obtained this result under the additional hypothesis that $\Phi$ is quadratically conditioned with respect to  $S=\argmin \Phi$.
We first establish some preliminary results of independent interest, and that will be useful in  proving the convergence result in the one-dimensional framework.

\begin{proposition}\label{conv-compact1}
Let $\cH$ be a finite dimensional Hilbert space, and let $\Phi : \mathcal H \rightarrow \mathbb R$ be a convex continuously differentiable function 
such that  $S= \argmin \Phi \neq \emptyset$.
Let $x:[ t_0  ;+\infty [ \rightarrow\mathcal{H} $ be a   solution  of $ \mbox{(AVD)}_{\alpha} $ with $\alpha = 3$. Then, 

\medskip

$i)$ \ $x$ is bounded and
$
\lim_{t\to +\infty} \mbox{\rm dist} (x(t), S)=0.
$

\medskip

$ii)$ Moreover, if $x(t) \in  S$ for $t$ large enough, then $x(t)$ converges to a point in $S$.
\end{proposition}
\begin{proof}
By Theorem \ref{fast}
 $$
 \lim_{t \rightarrow+\infty}\Phi(x(t))=\min\Phi.
 $$
 By Theorem \ref{speed-decay-pointwise}, for $\alpha \geq 3$, and in particular for $\alpha =3$, we have 
 $$
 \sup_{t\geq t_0} \Vert x(t)\Vert < + \infty.
$$
Then the result follows from a classical topological argument.
If $\mbox{\rm dist} (x(t), S)$ fails to converge to zero, this implies the existence of a sequence $t_n \to + \infty$ and  $\epsilon >0$ such that, for all $n \in \mathbb N$, $\mbox{\rm dist} (x(t_n), S) \geq \epsilon$. Since $(x(t_n))$ is bounded, and $\cH$ is a finite-dimensional Hilbert space, after extracting a convergent subsequence $x(t_{n_k}) \to \bar{x}$, we obtain $\mbox{\rm dist} (\bar{x}, S) \geq \epsilon$
and $\bar{x} \in S$ (a consequence of $x(t)$  minimizing), a clear contradiction.\\
Suppose moreover that $x(t) \in  S$ for $t \geq t_1$. Hence, $\nabla \Phi (x(t))=0$, which by $ \mbox{(AVD)}_{\alpha} $ gives
$\ddot{x}(t) + \frac{\alpha}{t} \dot{x}(t)  =0.$
Equivalently,
$
\frac{d}{dt}\left(t^{\alpha}\dot{x}(t)\right)=0,
$
which gives
$
\dot{x}(t)= \frac{C}{t^{\alpha}}.
$
Hence for $\alpha >1$, and in particular for $\alpha =3$, $\dot{x}$ is integrable, which implies the convergence of the trajectory.

\end{proof}

The following result is valid in a general Hilbert space, and for any $\alpha>0$. It  plays a key role in  the proof of the convergence result.

\begin{proposition}\label{basic.dim1} Let $\cH$ be a  Hilbert space, $\alpha \leq 3$, and $x(\cdot)$ be a trajectory of $ \mbox{(AVD)}_{\alpha} $. Suppose that for some $ t_2 \geq t_1$
$$
x(t_1)= x(t_2) \in S= \argmin \Phi.
$$
Then, 
$$
 t_2^{\frac{\alpha}{3}} \|\dot{x}(t_2) \|\leq t_1^{\frac{\alpha}{3}}\|\dot{x}(t_1)\|.
$$
In particular, for $\alpha =3$,
$$
 t_2 \|\dot{x}(t_2) \leq t_1\|\dot{x}(t_1)\|.
$$
\end{proposition}
\begin{proof}
Set $z= x(t_1)= x(t_2) \in S= \argmin \Phi$, take $p=\min (1,\frac{\alpha}{3})$, and consider
the function $\mathcal{E}_{\lambda, \xi}^{p}:[t_0, +\infty[ \to \mathbb R^+$
\begin{equation} \label{E:basic-Lyap-bb}
\mathcal{E}_{\lambda, \xi}^{p}(t)=t^{2p}\left[ \Phi(x(t))-\min_{\mathcal H} \Phi\right] + \frac{1}{2}\Vert \lambda(t)(x(t)-z)+t^p \dot{x}(t)\Vert^{2}+\frac{\xi(t)}{2}\Vert x(t)-z\Vert^2 
\end{equation}
which serves as a Lyapunov function in the proof of  Theorem \ref{fast}. It is nonincreasing, and hence $   \mathcal{E}_{\lambda, \xi}^{p}(t_2) \leq \mathcal{E}_{\lambda, \xi}^{p}(t_1)$. Because of  $z= x(t_1)= x(t_2)\in S= \argmin \Phi $, this is equivalent to
$$
\frac{1}{2}\Vert t_2^p \dot{x}(t_2)\Vert^{2} \leq \frac{1}{2}\Vert t_1^p \dot{x}(t_1)\Vert^{2}.
$$
By $p=\min (1,\frac{\alpha}{3})$, we get the announced result.
\end{proof}

The idea of the demonstration of the following Theorem is due to P. Redont (personal communication).

\begin{Theorem}
Take $\cH= \mathbb R$, and let $\Phi : \mathbb R \rightarrow \mathbb R$ be a convex continuously differentiable function 
such that  $S= \argmin \Phi \neq \emptyset$.
Let $x:[ t_0  ;+\infty [ \rightarrow\mathcal{H} $ be a    solution  of $ \mbox{(AVD)}_{\alpha} $ with $\alpha = 3$. Then $x(t)$ converges, as $t\to +\infty$, to a point in $S$.
\end{Theorem}
\begin{proof}
Recall that, by Theorem \ref{speed-decay-pointwise},  for $\alpha =3$, the trajectory is bounded, and minimizing.  As a consequence,  when $ \argmin \Phi  $ is reduced to a singleton $x^*$, it is the unique cluster point of the trajectory, which implies the convergence of the trajectory to $x^*$.
 Thus we consider the case where $ \argmin \Phi  $  is an interval of positive length (possibly infinite), let $ \argmin \Phi  = [a,b]$.
 We  consider the case where $a$ and $b$ are finite. The argument works in the same way when one of them, or both, is infinite. There are three possible cases:
 \begin{itemize}
 \item There exists  $T \geq t_0$ such that $x(t) \geq b$ \ for all $t \geq T$. Then $b$ is the unique cluster point of the trajectory, which implies the convergence of the trajectory to $b$. Symetrically, if 
there exists  $T \geq t_0$ such that $x(t) \leq a$, for all $t \geq T$, then the trajectory converges  to $a$.
 \item There exists  $T \geq t_0$ such that, for all $t \geq T$,  $ a \leq x(t) \leq b$. Then, the convergence of the trajectory is a consequence of Proposition \ref{conv-compact1} $ii)$.

\item It remains to  consider the case,  which is the most delicate to analyze,  where the trajectory passes in $ a $ and $ b $ an infinite number of times. Indeed, we will see that this is impossible. The argument is based on the analysis of the decay of the quantity $ t \|\dot{x}(t)\|$ during a loop.  
Let $s_{n} \leq t_{n} \leq u_n \leq v_n $ be consecutives times such that  $x(s_{n})=a$,  $x(t_{n})=b$, and $a \leq x(t) \leq b$ for all $t\in [s_{n} , t_{n} ]$, $x(u_{n})=b$, $x(v_{n})=a$ and
$a \leq x(t) \leq b$ for all $t\in [u_{n} , v_{n} ]$.
 For  $t\in [s_{n} , t_{n} ]$ we have 
$t\ddot{x}(t) + \alpha \dot{x}(t)  =0$. Equivalently
\begin{equation}\label{E:dim1}
\frac{d}{dt}\left(t\dot{x}(t)\right) + (\alpha -1) \dot{x}(t)  =0.
\end{equation}
After integration  of \eqref{E:dim1} on the interval $[s_{n} , t_{n} ]$, we obtain
$$
t_{n}\dot{x}(t_{n})- s_{n}\dot{x}(s_{n})= - (\alpha -1) (b-a).
$$
Taking account of the sign of the derivative of $x$    (positive at $s_{n}$, since the trajectory enters the interval $[s_{n} , t_{n} ]$, positive at $t_{n}$ since it leaves the interval), we have
$$
|t_{n}\dot{x}(t_{n})|= |s_{n}\dot{x}(s_{n})|- (\alpha -1) (b-a).
$$
Symetrically, 
$$
|v_{n}\dot{x}(v_{n}) |= | u_{n}\dot{x}(u_{n})| - (\alpha -1) (b-a).
$$
By Lemma \ref{basic.dim1} we have
$$
 | u_{n}\dot{x}(u_{n})| \leq |t_{n}\dot{x}(t_{n})| .
$$
Combining the above equalities, we obtain
$$
 | v_{n}\dot{x}(v_{n})| \leq |s_{n}\dot{x}(s_{n})| -2 (\alpha -1) (b-a) .
$$
 Therefore, each time the trajectory returns to  $a$ after passing  $b$,  the quantity $ t \|\dot{x}(t)\|$ decreases by a fixed positive quantity. This excludes the possibility of an infinite number of loops, and thus gives the contradiction.
 \end{itemize}
\end{proof}

\begin{remark}
Based on the convexity assumption on $\Phi$, and without any further geometrical assumption on $\Phi$, the convergence of the trajectories in the case $\alpha \leq 3$ is still an open problem.
The estimates we obtained in the previous sections did not allow us to conclude.
For example, let us consider  the case $ \alpha = 3 $.
Then, by Theorem \ref{speed-decay} 
 \begin{equation}\label{E:weak-conv-0}
\int_{t_0 }^{+\infty} t^{1-\epsilon} \Vert\dot{x}(t)\Vert^2 dt < + \infty  \quad \mbox{for all } \  \epsilon >0 . 
\end{equation}  
Let us examine how to exploit this information in the integration of inequation \eqref{E:weak-conv2}. First, after multiplication of \eqref{E:weak-conv2} by $t^{\alpha -1}$ and integration we obtain
$$
\dot{h}_z (t) \leq \frac{C}{t^{\alpha}} + \frac{1}{t^{\alpha}}\int_{t_0}^t  s^{\alpha}\Vert\dot{x}(s)\Vert^2 ds.
$$
After division by $t^{\epsilon}$, we get
$$
\frac{1}{t^{\epsilon}}\dot{h}_z (t) \leq \frac{C}{t^{\alpha +\epsilon}} + \frac{1}{t^{\alpha +\epsilon}}\int_{t_0}^t  s^{\alpha}\Vert\dot{x}(s)\Vert^2 ds.
$$
Integrating the above inequality, and using \eqref{E:weak-conv-0}  with Fubini theorem gives, for all $\epsilon >0$
\begin{equation}\label{E:weak-conv}
  \quad \left[\frac{1}{t^{\epsilon}} \dot{h}_z\right]^+ \in L^1 (t_0, + \infty).
\end{equation}
Moreover, by Theorem \ref{speed-decay-pointwise} item $(1)$, we know that the trajectory, and therefore $h_z$ is bounded and nonnegative. It may be asked whether these properties are sufficient to conclude that  $h_z$ converges, at least ergodically. Unfortunately, the answer to this question is negative. Take for example $$h (t)= 1+ \sin(\ln(t)).$$
 Then, $ 0 \leq h(t) \leq 2$ and $\dot{h}(t)= \frac{\cos(\ln(t))}{t}$. Hence
$|\frac{1}{t^{\epsilon}}\dot{h}_z (t)| \leq \frac{1}{t^{1+\epsilon}}$, and
\eqref{E:weak-conv} is satisfied. But $h(t)$ fails to converges,   even ergodically. Indeed, a direct computation gives
$$
\frac{1}{t-t_0}\int_{t_0}^{t} \sin(\ln(s))ds \sim \frac{1}{2}(\sin(\ln(t)-\cos(\ln(t))
$$
which clearly fails to converge as $t$ goes to $+\infty$.
Indeed, the proof of Theorem \ref{w-convergence} uses precisely  \eqref{E:weak-conv} for $\epsilon=0$.
\end{remark}

\bigskip

Thus, it is natural  to make additional geometric assumptions on $\Phi$ which guarantee the convergence of  the trajectories of $ \mbox{(AVD)}_{\alpha} $. Of particular interest is  the strong convergence property which occurs for example in the presence of strong convexity. That is the situation  we are now considering. For other cases of strong convergence (solution set with non-empty interior, even functions) one can consult \cite{ACPR}, \cite{AC1}.

\subsection{Strong convergence}
In this subsection, $\cH$ is a general Hilbert space, and we assume that the convex function $\Phi$ has a \textit{strong minimum}, i.e., there exist 
$x^*\in \cH$ and $\mu>0$ such that for every $x\in \cH$,
\begin{equation}\label{eq.strong_minimum}
\Phi(x)\geq \Phi(x^*)+\frac{\mu}{2}\|x-x^*\|^2.
\end{equation}
This implies clearly that $\argmin \Phi=\{x^*\}$. Under this condition, we are able to show the strong convergence of the trajectories to this unique minimum, and determine precisely the decay rate of the energy $W$ along the trajectories.
Indeed, an interesting property that has been put recentenly to the fore in   \cite{AC1}, \cite{ACPR}, \cite{SBC} is that the convergence rates increase indefinitely with larger values of $\alpha$ for these functions. 
The following theorem completes these results by examinating also the case $\alpha \leq 3$ and gives a synthetic view  of this situation.
For simplicity, we assume that the inequality \eqref{eq.strong_minimum} is satisfied for every $x\in \cH$, but all that follows can be readily extended to the case where this inequality is satisfied only in a neighborhood of $x^*$.

An important instance for which 
\eqref{eq.strong_minimum} holds is the class of strongly convex functions. We recall 
 that the function $\Phi : \mathcal H \rightarrow \mathbb R$ is {\it strongly convex}, if  there exists a positive constant $\mu$  such that
$$  
\Phi(y) \geq \Phi (x) + \langle  \nabla \Phi (x)  , y-x \rangle + \frac{\mu}{2} \|x-y  \|^2
$$
for all $x, y \in \mathcal H$. 
Of course if  $x^*$ is a minimizer of $\Phi$, then 
$\nabla \Phi(x^*)=0$  and we recover \eqref{eq.strong_minimum} from the above inequality.

\smallskip

\begin{Theorem}
Let $\Phi:\mathcal H\rightarrow\R$ be a convex function which admits a strong minimum $x^* \in\cH$. Let $x:[t_0,+\infty[\to\mathcal H$ be a solution trajectory of $ \mbox{(AVD)}_{\alpha} $ with $\alpha >0$. Then $x(t)$ converges strongly, as $t\to+\infty$, to the unique element $x^*\in\argmin\Phi$. Moreover

 \medskip

  $(i)\quad   \Phi (x(t))-\min_{\mathcal H} \Phi 
  = \mathcal O \left(\frac{1}{t^{\frac{2\alpha}{3}}}\right)$
  
 \medskip
 
 $(ii) \quad \|x(t)-x^*\|^2=\mathcal O\left(t^{-\dta}\right)$, 
 
 \medskip 
 
 $(iii) \quad \Vert\dot{x}(t)\Vert =\mathcal O (t^{-\frac{1}{3 }\alpha}).$

\end{Theorem}
\begin{proof}
The case $\alpha >3$ has already been analyzed  in the  articles mentioned above, see for example \cite[Theorem 3.4]{ACPR}.
The case $\alpha \leq 3$ is a direct consequence of Theorem \ref{fast} and Theorem \ref{speed-decay}, and of the property of strong minimum of $x^*$. Indeed, 
 we have  
$$\Vert x(t)-x^*\Vert^2\leq \frac{2}{\mu}\left(\Phi(x(t))-\Phi(x^*)\right).$$
The first two items then follow from Theorem \ref{fast} and the above inequality.
As a consequence $x(t)$ converges strongly to $x^*$, and so is bounded. Hence, we can apply Theorem \ref{speed-decay}, which gives the third item.
\end{proof}

\begin{remark}
The formulas for the rate of decay are the same in the cases $\alpha \leq 3$, and $\alpha \geq3$. It is not such surprising, since the proof used in \cite[Theorem 3.4]{ACPR} in the case $\alpha > 3$ is based on the following Lyapunov function
\begin{align*} 
\mathcal E_\lambda^p(t): =
  t^p\left(t^2(\Phi(x(t))-\min_{\mathcal H}\Phi)+ 
    \frac{1}{2}\|\lambda(x(t)-x^*)+t\dot x(t)\|^2\right),
\end{align*}
whose structure is quite similar to ours.
\end{remark}

\begin{remark} A model example of strongly convex function is  $\Phi(x) = \frac{1}{2} \|x\|^2$, which is a positive definite quadratic function. In this case, one can compute explicitely the solution trajectories of \eqref{edo01} with the help of the Bessel functions. Indeed
the solution of 
$$\mbox{(AVD)}_{\alpha} \quad \quad \ddot{x}(t) + \frac{\alpha}{t} \dot{x}(t) + x(t) = 0
$$
with  Cauchy data $x(0)= x_0$, \ $\dot{x}(0) =0$ is given by
\begin{equation*}
  x(t) =2^{\frac{\alpha -1}{2}} \Gamma(\frac{\alpha +1}{2})\frac{J_{\frac{\alpha -1}{2}}(t)}{t^{\frac{\alpha -1}{2}}}x_0.
\end{equation*}
In the above formula
$J_{\frac{\alpha -1}{2}}(\cdot)$ is the first kind
 Bessel function of order $\frac{\alpha -1}{2}$. For large $t$,
\begin{center}
$
J_{\alpha}(t)= \sqrt{\frac{2}{\pi t}} \left( \cos \left(t-  \frac{\pi \alpha}{2} -\frac{\pi}{4}\right) +\mathcal O(\frac{1}{t})\right). 
$
\end{center}
Hence
$$
\Phi(x(t))-\min_{\mathcal H}\Phi=\mathcal O(t^{-\alpha}).
$$
One can compare  with the rate $\mathcal O(t^{-\dta})$, which is valid for arbitrarily strongly convex functions.
One can also conclude from this example that, when $0< \alpha \leq 2$, \textit{the parameter of the optimal rate of convergence (in the worst case) is between $\frac{2\alpha}{3}$ and $\alpha$}.
\end{remark}

\begin{remark} The strong convexity property is a particular instance of the Kurdyka-Lojasiewicz property. It would be interesting to examine the problem of convergence for  convex functions satisfying this property. One can consult B\'egout-Bolte-Jendoubi \cite{BBJ} for a recent study of  damped second-order gradient systems under this geometric assumption. 
\end{remark}

\section{Inertial Forward-Backward algorithms}\label{algo}
Our study aims at complementing in the subcritical case $\alpha <3$ the results on the rate of convergence of the inertial forward-backward methods, and of the FISTA type algorithms.
These splitting algorithms aim at solving structured optimization problems
\begin{equation}\label{algo1}
\min \left\lbrace   \Phi (x) + \Psi (x) : \ x \in \mathcal H   \right\rbrace   
\end{equation}
where  $\Phi: \mathcal H \to  \mathbb R$ is a continuously differentiable convex function whose gradient is Lipschitz continuous, and $\Psi: \mathcal H \to  \mathbb R \cup \lbrace   + \infty  \rbrace $ is a proper lower-semicontinuous convex function. We set $\Theta:=\Phi +\Psi$, which is the convex lower-semicontinuous  function to minimize.
Based on the link between continuous dynamical systems and algorithms, we are naturally led to extend the dynamics studied in the previous sections, and to consider the differential inclusion
\begin{equation}\label{diff-incl}
 \mbox{(AVD)}_{\alpha} \quad \quad \ddot{x}(t) + \frac{\alpha}{t} \dot{x}(t) +\nabla \Phi (x(t)) + \partial \Psi (x(t))  \ni 0,
\end{equation}
where $\partial \Psi$ is the subdifferential of $\Psi$ in the sense of convex analysis. A detailed study of this differential inclusion goes far beyond the scope of the present article. See \cite{ACR} for some results in the case of a fixed positive damping parameter.
However, given the validity of the subdifferential inequality for convex functions, the (generalized) chain rule for derivatives over curves (see \cite{Bre1}), most of the results presented in the previous sections can be transposed to this more general context. Thus, the results  obtained in the continuous case will only serve as guidelines 
for the study of the associated algorithms that we present now.

 As shown in \cite{ACPR}, \cite{SBC}, the time discretization of  
this system, implicit with respect to the nonsmooth term $\partial \Psi$, and explicit with respect to the smooth term $\nabla \Phi$, gives the 
 the Inertial Forward-Backward algorithm
$$\mbox{\rm(IFB)}_{\alpha} \quad\left\{ 
\begin{array}{l}
y_k=x_k+(1 - \frac{\alpha}{k})(x_k-x_{k-1})\\
\rule{0pt}{10pt}
x_{k+1}=\prox_{s\Psi}(y_k-s\nabla \Phi(y_k)).
\end{array}
\right.$$
Note that other types of discretization give slightly different  inertial forward-backward algorithms, see \cite{LFP}, an interesting topic for further research.
In the above formula, the parameter $s$ enters as $s = h^2 $, where $h$ is the step size of the discretization, and 
$\mbox{prox}_{ \gamma \Psi }: \cH \to \cH$ is the classical proximal operator. Recall that, for any $\gamma >0$, for any $x\in  \cH$,  
\begin{equation}\label{algo6}
\mbox{prox}_{ \gamma \Psi } (x)= {\argmin}_{\xi \in  \mathcal H} \left\lbrace \Psi (\xi) + \frac{1}{2 \gamma} \| \xi -x  \|^2
\right\rbrace  = \left(I + \gamma \partial \Psi \right)^{-1} (x).
\end{equation}
This last equality expresses that $\mbox{prox}_{ \gamma \Psi } $ is the resolvent of index $\gamma$ of the maximal monotone operator $\partial \Psi$.
One can consult \cite{BC}, \cite{PB}, \cite{Pey}, \cite{Pey_Sor} for a recent account on the proximal-based splitting methods.
 
\smallskip

a) For $\alpha = 3$, we recover the classical FISTA algorithm developed by Beck-Teboulle \cite{BT}, based on the acceleration method introduced by Nesterov \cite{Nest1} in the smooth case, and by G\"uler \cite{Guler} in the proximal setting. The rate of convergence of this method is $\Theta(x_k)-\min_\mathcal H \Theta= \mathcal O \left(\frac{1}{k^2}\right)$, where $k$ is the number of iterations.
Convergence of the sequences generated by Nesterov's accelerated gradient method and of FISTA, has been an elusive question for decades. 

\smallskip

b) For $\alpha >3$, it was shown by Chambolle-Dossal \cite{CD}   that each sequence generated by  $\mbox{\rm(IFB)}_{\alpha}$  converges weakly to an optimal solution, and  Attouch-Peypouquet \cite{AP} obtained the improved rate of convergence $\Theta(x_k)-\min_\mathcal H \Theta = o (\frac{1}{k^2})$. 

\smallskip

c) The case $\alpha <3$ has remained largely unknown. Precisely, we will study this situation and show  results parallel to  those obtained in the continuous framework in the previous sections.

\smallskip

We will systematically assume the following set of hypotheses
$$(H) \left\{
\begin{array}{l}
\bullet \cH \mbox{ \ is a real Hilbert space};\\
\bullet \ \Phi: \cH \to \R \mbox{ is convex,  differentiable with $L$-Lipschitz continuous gradient;}\\
\bullet \ \Psi: \cH \to \rinf \mbox{ is convex, proper and lower semicontinuous;}\\
\bullet  \mbox{ $\Theta:=\Phi + \Psi$ has a nonempty set of minimizers: $S=\argmin \Theta\neq \emptyset$;}\\
\bullet \  \mbox{The parameter $\alpha$ is positive;}
\\
\bullet \  \mbox{The parameter $s$ satisfies $s\in \,]0,1/L]$.}
\end{array}
\right.$$
Our main result in the algorithmic framework is presented below.

\begin{Theorem}\label{fast-algo} 
Let us make assumptions $(H)$.
Let $(x_k)$ be a sequence of iterates generated by algorithm $\mbox{\rm(IFB)}_{\alpha}$.
Suppose $0 < \alpha \leq 3$. Then the following  convergence result is verified: for all $p < \frac{2\alpha}{3}$, for all $k \geq 1$
$$
 (\Phi + \Psi) (x_k)-\min_{\mathcal H} (\Phi + \Psi) =   \mathcal O (\frac{1}{k^{p}}).
$$
\end{Theorem}

\begin{remark}
It is probable that the same order of convergence as in the continuous case is valid, $(\Phi + \Psi) (x_k)-\min_{\mathcal H} (\Phi + \Psi) =  \mathcal O \left( \frac{1}{k^{\frac{2\alpha}{3}}} \right)$. It is for technical reasons, in order to make the proof not too complicated, that we obtained the sligthly weaker result as indicated above. 
\end{remark}

\subsection{The inertial forward-backward algorithm: classical facts}
Let us first rewrite $\mbox{\rm(IFB)}_{\alpha}$ in a more compact way. 
 Let us define the operator $G_s:\mathcal H \to \mathcal H$ by 
$$ G_s(y)=\frac{1}{s}\left(y- \prox_{s\Psi}(y-s\nabla \Phi(y))\right),$$
and set $\alpha_k = 1 -  \frac{\alpha}{k}$. 
Thus, the algorithm can be written in an equivalent way as
$$\mbox{\rm(IFB)}_{\alpha} \quad\left\{ 
\begin{array}{l}
y_k=x_k+\alpha_k(x_k-x_{k-1})\\
\rule{0pt}{10pt}
x_{k+1}=y_k-sG_s(y_k).
\end{array}
\right.$$

The following  classical lemma plays a crucial role in the convergence analysis of the forward-backward algorithms.    It is  known as the descent rule for the forward-backward methods, see \cite [Lemma 2.3] {BT}, \cite [Lemma 1] {CD}.
Its validity relies essentially on the  stepsize limitation     
 $s \leq \frac {1} {L} $.
\begin{lemma}\label{lm.properties_G_s}
Assume hypothesis $(H)$. Then
\begin{itemize}
\item[$(i)$] For all $x,y \in \cH$
\begin{equation}\label{eq.property_Theta}
\Theta(y-sG_s(y))\leq \Theta(x)+\langle G_s(y),y-x\rangle -\frac{s}{2}\|G_s(y)\|^2.
\end{equation}
\item[$(ii)$] The  operator $G_s$ is monotone, and the following equivalences hold true
$$z\in \argmin \Theta\,\Longleftrightarrow\,  G_s(z)=0.$$ 
\end{itemize}
\end{lemma}

\subsection{Design of the Lyapunov function}
Let us give a discrete version of the  Lyapunov function 
\begin{equation} \label{E:basic-Lyap-b}
\mathcal{E}_{\lambda, \xi}^{p}(t)=t^{2p}\left[ \Theta(x(t))-\min_{\mathcal H} \Theta\right] + \frac{1}{2}\Vert \lambda(t)(x(t)-z)+t^p \dot{x}(t)\Vert^{2}+\frac{\xi(t)}{2}\Vert x(t)-z\Vert^2 ,
\end{equation}
that has been used in the continuous case (where now the function to minimize is $\Theta$).
To that end, let us reformulate $\mathcal{E}_{\lambda, \xi}^{p}$  with the help of the energy function
$W(t) = \Theta(x(t))-\min_{\mathcal H} \Theta +\frac{1}{2} \Vert\dot{x}(t)\Vert^{2}$, and the anchor function $h_z (t)= \frac{1}{2}\Vert x(t)-z \Vert^2 $.
After development of \eqref{E:basic-Lyap-b}, we get
\begin{equation*} 
\mathcal{E}_{\lambda, \xi}^{p}(t)=t^{2p}\left[ \Theta(x(t))-\min_{\mathcal H} \Theta + \frac{1}{2}\Vert\dot{x}(t)\Vert^{2}\right] + \frac{1}{2}\lambda(t)^2\Vert x(t)-z \Vert^2  + \lambda(t)t^p \langle x(t)-z,\dot{x}(t)\rangle +   \frac{\xi(t)}{2}\Vert x(t)-z\Vert^2 .
\end{equation*}
Hence,
\begin{equation} \label{E:basic-Lyap-d}
\mathcal{E}_{\lambda, \xi}^{p}(t)=t^{2p}W(t) +  \lambda(t)t^p \dot{h}(t) + \left( \lambda(t)^2 + \xi(t) \right)h(t)  .
\end{equation}
Let us introduce the  global energy (potential + kinetic) at stage  $k \in \N^*$ 
$$W_k :=\Theta(x_k)-\min \Theta + \frac{1}{2s}\|x_k-x_{k-1}\|^2 ,$$
and  the anchor function to the solution set (where $z$ denotes an element of $S$),
$$
h_k := \demi \|x_k-z\|^2 .
$$
Thus, a natural candidate for the Lyapunov function in the discrete case would be
\begin{equation} \label{E:basic-Lyap-f}
\mathcal{E}_k= s^pk^{2p} W_k+ s^{\frac{p-1}{2}}\lambda_k k^p(h_k - h_{k-1})+ (\lambda_k^2 + \xi_k)h_{k-1}  
\end{equation}
where $\lambda_k$ and $\xi_k$ are positive parameters that will be defined further.
Let us verify that  $\mathcal{E}_k$ is non-negative. Indeed,
\begin{equation*} 
\mathcal{E}_k= s^pk^{2p} (\Theta(x_k)-\min \Theta)+   s^{p-1}\frac{k^{2p}}{2} \|x_k-x_{k-1}\|^2  +    \demi s^{\frac{p-1}{2}}\lambda_k k^p( \|x_k-z\|^2 -  \|x_{k-1}-z\|^2)+ \frac{\lambda_k^2}{2} \|x_{k-1}-z\|^2 +  \frac{\xi_k}{2} \|x_{k-1}-z\|^2
\end{equation*}
From
$$
\|x_k-z\|^2 -  \|x_{k-1}-z\|^2 \geq 
2\langle x_{k-1}-z, x_k-x_{k-1}\rangle 
$$
and $\xi_k \geq 0  $   we infer
\begin{align*} 
\mathcal{E}_k &\geq  s^pk^{2p} (\Theta(x_k)-\min \Theta)+  \frac{1}{2} \left[ s^{p-1}k^{2p} \|x_k-x_{k-1}\|^2  +  2 s^{\frac{p-1}{2}} \lambda_k k^p\langle  x_k-x_{k-1}, x_{k-1}-z\rangle + \lambda_k^2 \|x_{k-1}-z\|^2 \right]\\
& = s^p k^{2p} (\Theta(x_k)-\min \Theta)+
\frac{1}{2} \|\lambda_k (x_{k-1} -z) + s^{\frac{p-1}{2}}k^p (x_k-x_{k-1})\|^2 .
\end{align*}
As a result, the condition  $\xi_k \geq 0$ ensures that 
$\mathcal{E}_k$ is non-negative, and minorized by $s^p k^{2p} (\Theta(x_k)-\min \Theta)$.
Thus, if we can prove that $\mathcal{E}_k$ is  bounded from above,  we will get the desired result.
In the continuous case, we have proved that, with judicious choices of the parameters $\lambda(t)$ and $\xi(t)$ the function $\mathcal{E}_{\lambda, \xi}^{p}$ was non-increasing. The discretization induces certain additional terms that we must show to be negligible, which makes the proof more technical. Indeed, we will
have to argue with the slightly modified function
$$
\mathcal{\tilde{E}}_k := \mathcal{E}_k   - \frac{p}{s^{1-p}}(k)^{2p-1}\|x_k-x_{k-1}\|^2
$$
which involves a lower order correction term.

\medskip

Our first lemmas  are based  on the study carried out by Attouch-Cabot \cite{AC2} in the case of a general coefficient $\alpha_k$. They analyze successively the rate of decay of the energy $W_k$ and of the anchor functions $h_k$. Then, based on formula \eqref{E:basic-Lyap-f} we will put these results together.
 
\subsection{Decay of the energy}
Given a sequence of iterates $(x_k)$ generated by algorithm $\mbox{\rm(IFB)}_{\alpha}$, 
let us evaluate the  decay of the   energy 
$W_k :=\Theta(x_k)-\min \Theta + \frac{1}{2s}\|x_k-x_{k-1}\|^2.$

\begin{proposition}\label{pr.energy_decay}
Under hypothesis $(H)$, let $(x_k)$ be a sequence generated by algorithm $\mbox{\rm(IFB)}_{\alpha}$. The energy sequence $(W_k)$ satisfies for every $k\geq 1$,
\begin{equation}\label{eq.energy_decay}
W_{k+1}-W_k\leq -\frac{1-\alpha_k^2}{2s}\|x_k-x_{k-1}\|^2.
\end{equation}
As a consequence, the sequence $(W_k)$ is nonincreasing.
\end{proposition}
\begin{proof}
By applying formula (\ref{eq.property_Theta}) with $y=y_k$ and $x=x_k$, we obtain
\begin{equation}\label{eq.majo_theta(x_k+1)}
\Theta(x_{k+1})=\Theta(y_k-sG_s(y_k))\leq \Theta(x_k)+\langle G_s(y_k),y_k-x_k\rangle -\frac{s}{2}\|G_s(y_k)\|^2.
\end{equation}
Reformulate the second member of \eqref{eq.majo_theta(x_k+1)} using the sequence $ (x_k) $. By  definition of $G_s$ and  $\mbox{\rm(IFB)}_{\alpha}$ we have
\begin{eqnarray*}
\langle G_s(y_k),y_k-x_k\rangle&=& -\frac{1}{s}\langle x_{k+1}-y_k,y_k-x_k\rangle\\
&=&-\frac{1}{s}\langle x_{k+1}-x_k-\alpha_k(x_k-x_{k-1}),\alpha_k(x_k-x_{k-1})\rangle\\
&=&-\frac{\alpha_k}{s}\langle x_{k+1}-x_k,x_k-x_{k-1}\rangle+\frac{\alpha_k^2}{s}\|x_k-x_{k-1}\|^2
\end{eqnarray*}
and
\begin{eqnarray*}
\|G_s(y_k)\|^2&=&\frac{1}{s^2}\|x_{k+1}-y_k\|^2=\frac{1}{s^2}\|x_{k+1}-x_k-\alpha_k(x_k-x_{k-1})\|^2\\
&=&\frac{1}{s^2}\|x_{k+1}-x_k\|^2+\frac{\alpha_k^2}{s^2}\|x_k-x_{k-1}\|^2-\frac{2\alpha_k}{s^2}\langle x_{k+1}-x_k,x_k-x_{k-1}\rangle.
\end{eqnarray*}
In view of (\ref{eq.majo_theta(x_k+1)}), we obtain
$$\Theta(x_{k+1})\leq \Theta(x_k)+\frac{\alpha_k^2}{2s}\|x_k-x_{k-1}\|^2-\frac{1}{2s}\|x_{k+1}-x_{k}\|^2,$$
which can be equivalently rewritten as \eqref{eq.energy_decay}.
\end{proof}

\subsection{Anchor}
As a fundamental tool, we will use the distance to equilibria  to anchor the trajectory to the
solution set $S=\argmin \Theta$. 
To this end, given  $z\in \argmin \Theta$, recall that   $h_k=\demi \|x_k-z\|^2$. The next result can be found for example in \cite{AC2}. It plays a central role in our Lyapunov analysis of  $\mbox{\rm(IFB)}_{\alpha}$.
\begin{proposition}
\label{pr.h_k+1-h_k-alpha_k(h_k-h_k-1)}
Under $(H)$, we have 
\begin{equation}\label{eq.h_k+1-h_k-alpha_k(h_k-h_k-1)_egalite}
h_{k+1}-h_k-\alpha_k(h_k-h_{k-1})\leq \demi (\alpha_k^2+\alpha_k)\|x_k-x_{k-1}\|^2-s(\Theta(x_{k+1})-\min \Theta).
\end{equation}
\end{proposition}
\begin{proof}
Observe that
\begin{eqnarray*}
\|y_k-z\|^2&=&\|x_k+\alpha_k(x_k-x_{k-1})-z\|^2\\
&=&\|x_k-z\|^2+\alpha_k^2\|x_k-x_{k-1}\|^2+2\alpha_k\langle x_k-z,x_k-x_{k-1}\rangle\\
&=&\|x_k-z\|^2+\alpha_k^2\|x_k-x_{k-1}\|^2 \\
&+&\alpha_k \|x_k-z\|^2+ \alpha_k\|x_k-x_{k-1}\|^2-\alpha_k\|x_{k-1}-z\|^2\\
&=&\|x_k-z\|^2+\alpha_k(\|x_k-z\|^2-\|x_{k-1}-z\|^2)+(\alpha_k^2+\alpha_k)\|x_k-x_{k-1}\|^2\\
&=&2[h_k+\alpha_k(h_k-h_{k-1})]+(\alpha_k^2+\alpha_k)\|x_k-x_{k-1}\|^2.
\end{eqnarray*}
Setting briefly $H_k=h_{k+1}-h_k-\alpha_k(h_k-h_{k-1}) $, we deduce that
\begin{eqnarray*}
H_k&=&\demi\|x_{k+1}-z\|^2-\demi \|y_k-z\|^2+\demi (\alpha_k^2+\alpha_k)\|x_k-x_{k-1}\|^2\\
&=&\left\langle x_{k+1}-y_k,\demi(x_{k+1}+y_k)-z\right\rangle+\demi(\alpha_k^2+\alpha_k)\|x_k-x_{k-1}\|^2\\
&=&\langle x_{k+1}-y_k,y_k-z\rangle+\demi \|x_{k+1}-y_k\|^2+\demi(\alpha_k^2+\alpha_k)\|x_k-x_{k-1}\|^2.
\end{eqnarray*}
Using the equality $x_{k+1}=y_k-sG_s(y_k)$, we obtain (\ref{eq.h_k+1-h_k-alpha_k(h_k-h_k-1)_egalite}).
Inequality (\ref{eq.property_Theta}) applied with $y=y_k$ and $x=z$ yields
$$\Theta(x_{k+1})=\Theta(y_k-sG_s(y_k))\leq \Theta(z)+\langle G_s(y_k),y_k-z\rangle -\frac{s}{2}\|G_s(y_k)\|^2.$$
Since $\Theta(z)=\min \Theta$, we infer that
$$-s\langle G_s(y_k),y_k-z\rangle +\frac{s^2}{2}\|G_s(y_k)\|^2\leq -s(\Theta(x_{k+1})-\min \Theta),$$
which completes the proof of Proposition \ref{pr.h_k+1-h_k-alpha_k(h_k-h_k-1)}.
\end{proof}

\subsection{Proof of Theorem \ref{fast-algo}}
\begin{proof}
Let us evaluate from above $\mathcal{E}_{k+1} -\mathcal{E}_k$,
where $\mathcal{E}_k$ is given by \eqref{E:basic-Lyap-f}. We have 
\begin{equation} \label{E:Lyap-10}
\mathcal{E}_{k+1} -\mathcal{E}_k= s^p A_k + s^{\frac{p-1}{2}}B_k + C_k
\end{equation}
where
\begin{align*} 
A_k &= (k+1)^{2p} W_{k+1} -k^{2p} W_k \\
 B_k &= \lambda_{k+1} (k+1)^p(h_{k+1} - h_{k}) -  \lambda_k k^p(h_k - h_{k-1}) \\ 
 C_k &= (\lambda_{k+1}^2 + \xi_{k+1})h_{k} -  (\lambda_k^2 + \xi_k)h_{k-1} .     
\end{align*}
Let us examine successively each of the distinctive terms  $A_k, B_k, C_k$. We have
\begin{equation} \label{E:Lyap-11}
A_k = ((k+1)^{2p}-k^{2p}) W_{k+1} +
k^{2p} (W_{k+1}-W_k ).
\end{equation}
When $p \geq \frac{1}{2}$ we have
$(k+1)^{2p}-k^{2p} \leq 2p (k+1)^{2p-1}  $, and
when $p \leq \frac{1}{2}$ 
we have
$(k+1)^{2p}-k^{2p} \leq 2p k^{2p-1}  $. These are equivalent quantities, which lead to similar computations. Thus, in the following we suppose that $p \geq \frac{1}{2}$. The analysis in the case $p \leq \frac{1}{2}$  is quite similar. Hence we have  $(k+1)^{2p}-k^{2p} \leq 2p (k+1)^{2p-1}  $.
From Proposition \ref{pr.energy_decay} we have 
$W_{k+1}-W_k\leq -\frac{1-\alpha_k^2}{2s}\|x_k-x_{k-1}\|^2.$
Putting these two results together, and by definition of $W_{k+1}$,  we infer
\begin{equation} \label{E:Lyap-12}
A_k \leq  2p (k+1)^{2p-1}  (   \Theta(x_{k+1})-\min \Theta + \frac{1}{2s}\|x_{k+1}-x_{k}\|^2) -\frac{1-\alpha_k^2}{2s}k^{2p} \|x_k-x_{k-1}\|^2 .
\end{equation} 
Let us now consider $B_k$.
\begin{equation*} 
B_k = \lambda_{k+1} (k+1)^p(h_{k+1} - h_{k}) -  
\alpha_k ( h_{k} - h_{k-1} ))+
( \alpha_k   \lambda_{k+1} (k+1)^p             -\lambda_k k^p )(h_k - h_{k-1}).
\end{equation*}
From Proposition 
\ref{pr.h_k+1-h_k-alpha_k(h_k-h_k-1)}, and $0 \leq \alpha_k \leq 1$,
we infer
\begin{equation} \label{E:Lyap-13}
B_k \leq  \lambda_{k+1} (k+1)^p \left( \|x_k-x_{k-1}\|^2-s(\Theta(x_{k+1})-\min \Theta)    \right)+
( \alpha_k   \lambda_{k+1} (k+1)^p             -\lambda_k k^p )(h_k - h_{k-1}).
\end{equation}
Let us finally consider $C_k$.
\begin{equation} \label{E:Lyap-14}
C_k = (\lambda_{k+1}^2 + \xi_{k+1})(h_{k} - h_{k-1}) +( (\lambda_{k+1}^2 + \xi_{k+1})- (\lambda_k^2 + \xi_k))h_{k-1}.
\end{equation}
Let us put together \eqref{E:Lyap-12}, \eqref{E:Lyap-13}, and  \eqref{E:Lyap-14}. We obtain
\begin{align*} 
\mathcal{E}_{k+1} -\mathcal{E}_k &\leq 2p s^p(k+1)^{2p-1}  (   \Theta(x_{k+1})-\min \Theta + \frac{1}{2s}\|x_{k+1}-x_{k}\|^2) -\frac{1-\alpha_k^2}{2s}s^p
k^{2p}\|x_k-x_{k-1}\|^2 \\
&  +\lambda_{k+1} (k+1)^p s^{\frac{p-1}{2}}\left( \|x_k-x_{k-1}\|^2-s(\Theta(x_{k+1})-\min \Theta)    \right)+
s^{\frac{p-1}{2}}( \alpha_k   \lambda_{k+1} (k+1)^p  -\lambda_k k^p )(h_k - h_{k-1})    \\
&+(\lambda_{k+1}^2 + \xi_{k+1})(h_{k} - h_{k-1}) +( (\lambda_{k+1}^2 + \xi_{k+1})- (\lambda_k^2 + \xi_k))h_{k-1} .
\end{align*}
Let us reorganize the above expression in a parallel way to the formula below obtained in the continuous case, and that we recall below 
\begin{equation*} 
\begin{array}{lll}
\frac{d}{dt}\mathcal{E}_{\lambda, \xi}^{p}(t)&\leq& t^p \left[ 2p  t^{p-1}-\lambda(t))\right] \left(  \Phi(x(t))-\min_{\mathcal H} \Phi\right)  
+ \left[ \xi (t)+t^{p}\dot{\lambda}(t)-(\alpha -p)t^{p-1}\lambda(t)+\lambda(t)^{2}\right] \dot{h}(t) 
\\
&-& t^{p}\left[ (\alpha -p)t^{p-1}-\lambda(t)\right] \Vert \dot{x}(t)\Vert^{2} + \left[ \lambda(t)\dot{\lambda}(t)+\frac{\dot{\xi}(t)}{2}\right] \Vert x(t)-z\Vert^{2}. 
\end{array}
\end{equation*}
We obtain
\begin{align} \label{E:basic-ineq-algo1}
\mathcal{E}_{k+1} -\mathcal{E}_k &\leq 
(k+1)^p \left[ 2p s^p (k+1)^{p-1}-s^{\frac{p+1}{2}}\lambda_{k+1}\right] \left(  \Theta(x_{k+1})-\min \Theta \right) \\
&  +\left[ \lambda_{k+1}^2 + \xi_{k+1} +s^{\frac{p-1}{2}} (\alpha_k   \lambda_{k+1} (k+1)^p  -\lambda_k k^p) \right] (h_{k} - h_{k-1})\nonumber \\
&+\left[\lambda_{k+1} (k+1)^p s^{\frac{p-1}{2}} -\frac{1-\alpha_k^2}{2s} s^p k^{2p} \right] \|x_k-x_{k-1}\|^2 + \frac{p}{s^{1-p}}(k+1)^{2p-1} \|x_{k+1}-x_{k}\|^2 \nonumber\\
&+\left[ (\lambda_{k+1}^2 + \xi_{k+1})- (\lambda_k^2 + \xi_k)\right]h_{k-1}.\nonumber
\end{align}
Let us  analyze successively  the four terms which enter the second member of \eqref{E:basic-ineq-algo1}:

\smallskip

1) Let us make the two first terms of the second member
of \eqref{E:basic-ineq-algo1} equal to zero by taking
respectively
\begin{equation}\label{E:basic-ineq-algo2}
\lambda_{k}= \frac{2p}{s^{\frac{1-p}{2}}}k^{p-1}
\end{equation}
and 
$$
\xi_{k+1}= - \lambda_{k+1}^2  - s^{\frac{p-1}{2}} (\alpha_k   \lambda_{k+1} (k+1)^p  -\lambda_k k^p).
$$
This last equality gives equivalently
$$
\xi_{k+1}= -\lambda_{k+1}^2 - s^{\frac{p-1}{2}} \alpha_k   \lambda_{k+1} (k+1)^p  + s^{\frac{p-1}{2}}\lambda_k k^p.
$$
By \eqref{E:basic-ineq-algo2} we deduce that
\begin{align*}
\xi_{k+1}&= -\frac{4p^2}{s^{1-p}}(k+1)^{2p-2} - \frac{2p}{s^{1-p}} (1-\frac{\alpha}{k}) (k+1)^{2p -1} + \frac{2p}{s^{1-p}} k^{2p-1}\\
&= \frac{2p}{s^{1-p}} \left[-2p (k+1)^{2p-2} -(1-\frac{\alpha}{k}) (k+1)^{2p -1} + k^{2p-1}  \right]:= \frac{2p}{s^{1-p}}  D_k .
\end{align*}
Let us analyze the sign of $D_k$ by making an asymptotic development.
\begin{align}\label{E: dk}
D_k&= -2p (k+1)^{2p-2} - \left[ (k+1)^{2p -1} - k^{2p-1} \right]  + \frac{\alpha}{k} (k+1)^{2p -1} \\
&= -2p k^{2p-2}(1+\frac{1}{k} )^{2p-2} - k^{2p-1} \left[ (1+ \frac{1}{k})^{2p -1} - 1 \right]  + \alpha k^{2p -2} (1+\frac{1}{k})^{2p-1}\nonumber \\
&= -2p k^{2p-2}\left[ 1 + \frac{2p-2}{k} +\mbox{{\tiny O}}(\frac{1}{k})\right] - k^{2p-1} \left[ \frac{2p-1}{k} + \mbox{{\tiny O}}(\frac{1}{k})\right]  + \alpha k^{2p -2} \left[ 1 + \frac{2p-1}{k} + \mbox{{\tiny O}}(\frac{1}{k})\right] \nonumber \\
&  \sim (\alpha -4p +1  ) k^{2p-2}. \nonumber
\end{align}
Hence, the condition $\xi_k \geq 0$ is satisfied for
$$
p < \frac{\alpha + 1}{4}.
$$
It is the same as in the continuous case, except that now the inequality is  assumed to be strict.

\smallskip

2) Let us now consider
\begin{align*}
F_k:=& \left[\lambda_{k+1} (k+1)^p s^{\frac{p-1}{2}} -\frac{1-\alpha_k^2}{2s} s^p k^{2p} \right] \|x_k-x_{k-1}\|^2 + \frac{p}{s^{1-p}}(k+1)^{2p-1} \|x_{k+1}-x_{k}\|^2 \\
=& \left[\lambda_{k+1} (k+1)^p s^{\frac{p-1}{2}} -\frac{1-\alpha_k^2}{2s} s^p k^{2p} + \frac{p}{s^{1-p}}(k+1)^{2p-1}\right] \|x_k-x_{k-1}\|^2 \\
&+ \frac{p}{s^{1-p}}(k+1)^{2p-1}\left[ \|x_{k+1}-x_{k}\|^2 - \|x_k-x_{k-1}\|^2 \right] \\
\leq & \left[\lambda_{k+1} (k+1)^p s^{\frac{p-1}{2}} -\frac{1-\alpha_k^2}{2s} s^p k^{2p} + \frac{p}{s^{1-p}}(k+1)^{2p-1}\right] \|x_k-x_{k-1}\|^2 \\
&+ \frac{p}{s^{1-p}}\left[ (k+1)^{2p-1}\|x_{k+1}-x_{k}\|^2 - k^{2p-1}\|x_k-x_{k-1}\|^2 \right],
\end{align*}
where in the last inequality we have used $k^{2p-1} \leq (k+1)^{2p-1}  $ for $p \geq \frac{1}{2}$.
By \eqref{E:basic-ineq-algo2}, and the definition of $\alpha_k$ we have
\begin{align*}
\lambda_{k+1} (k+1)^p s^{\frac{p-1}{2}} -\frac{1-\alpha_k^2}{2s} s^p k^{2p} + \frac{p}{s^{1-p}}(k+1)^{2p-1}&\sim  \frac{2p}{s^{\frac{1-p}{2}}}(k+1)^{p-1} (k+1)^p s^{\frac{p-1}{2}} -\frac{1}{2s} \frac{2\alpha}{k}s^p k^{2p} + \frac{p}{s^{1-p}}(k+1)^{2p-1}\\
 & \sim \frac{1}{s^{1-p}}\left[3p - \alpha        \right] k^{2p-1}.
\end{align*}
As a consequence, for $p < \frac{\alpha}{3}$ we have 
$$F_k \leq \frac{p}{s^{1-p}}\left[ (k+1)^{2p-1}\|x_{k+1}-x_{k}\|^2 - k^{2p-1}\|x_k-x_{k-1}\|^2 \right].$$

\smallskip

3) Let us finally consider
$$
G_k:= (\lambda_{k+1}^2 + \xi_{k+1})- (\lambda_k^2 + \xi_k)= (\lambda_{k+1}^2 - \lambda_k^2 )+ (\xi_{k+1}-  \xi_k).
$$
Let us compute an equivalent of the discrete derivative $\lambda_{k+1} - \lambda_k $.
We have
\begin{align*}
\lambda_{k+1} - \lambda_k&= \frac{2p}{s^{\frac{1-p}{2}}} \left( (k+1)^{p-1} -k^{p-1} \right)\\
&= \frac{2p}{s^{\frac{1-p}{2}}} k^{p-1} \left( (1+\frac{1}{k})^{p-1} -1 \right)\\
&  \sim \frac{2p(p-1)}{s^{\frac{1-p}{2}}} k^{p-2}.
\end{align*}
Hence
\begin{align*}
\lambda_{k+1}^2 - \lambda_k^2&= (\lambda_{k+1} + \lambda_k )(\lambda_{k+1} - \lambda_k )\\
& \sim \frac{2p(p-1)}{s^{\frac{1-p}{2}}} k^{p-2} \times 2\frac{2p}{s^{\frac{1-p}{2}}}k^{p-1} \\
& \sim \frac{8p^2 (p-1)}{s^{1-p}}k^{2p-3} 
\end{align*}
On the other hand, from $\xi_{k+1} = \frac{2p}{s^{1-p}}  D_k $ and $D_k  \sim (\alpha -4p +1  ) k^{2p-2}$
(see \eqref{E: dk}) we immediately infer

$$\xi_{k+1}-  \xi_k  
  \sim \frac{2p(2p-2)}{s^{1-p}}(\alpha -4p +1  ) k^{2p-3}.$$
Hence (when the equivalent is not zero!)
\begin{align*}
G_k & \sim \frac{8p^2 (p-1)}{s^{1-p}}k^{2p-3}
+ \frac{2p(2p-2)}{s^{1-p}}(\alpha -4p +1  ) k^{2p-3}\\
&= 4p(p-1)(\alpha - 2p +1 )k^{2p-3}. 
\end{align*}
The same argument as in the continuous case leads us to assume $0 < p<1$. The condition $\alpha - 2p +1 >0$ is already implied by the previous assumption
$p \leq \frac{\alpha + 1}{4}.$\\
We have now all the ingredients to conclude. 
By taking $p < \min\{1, \frac{\alpha}{3}\}$, we have obtained that for $k$ sufficiently large
$$
\mathcal{E}_{k+1}-\mathcal{E}_k \leq \frac{p}{s^{1-p}} (k+1)^{2p-1}\|x_{k+1}-x_{k}\|^2 - \frac{p}{s^{1-p}}k^{2p-1}\|x_k-x_{k-1}\|^2 .
$$
Equivalently $\mathcal{E}_k - \frac{p}{s^{1-p}}k^{2p-1}\|x_k-x_{k-1}\|^2 $ is non-increasing, and hence is bounded from above. Returning to the definition of $\mathcal{E}_k$, this gives
\begin{align*} 
s^pk^{2p} (\Theta(x_k)-\min \Theta) &+    \demi s^{\frac{p-1}{2}}\lambda_k k^p( \|x_k-z\|^2 -  \|x_{k-1}-z\|^2)+ \frac{\lambda_k^2}{2} \|x_{k-1}-z\|^2 +  \frac{\xi_k}{2} \|x_{k-1}-z\|^2 \\
&+  s^{p-1}\frac{k^{2p}}{2} \|x_k-x_{k-1}\|^2 \left[ 1- \frac{2p}{k} \right]\leq C
\end{align*}
for some positive constant $C$.
From the definition of  $\lambda_k = \frac{2p}{s^{\frac{1-p}{2}}}k^{p-1} $ and $\xi_k \sim \frac{2p}{s^{1-p}}(\alpha -4p +1  ) k^{2p-2} $ (recall that $\alpha -4p +1 >0$), we deduce that, for some $\epsilon >0$, and $k$ large enough
$$\xi_k \geq \epsilon \lambda_k^2 . $$ 
On the other hand, for $k$ large enough
$$
s^{p-1}\frac{k^{2p}}{2} \|x_k-x_{k-1}\|^2 \left[ 1- \frac{2p}{k} \right] \geq \frac{1}{2(1+ \epsilon)} s^{p-1}k^{2p} \|x_k-x_{k-1}\|^2  .
$$
Combining the above results, and using
$$
\|x_k-z\|^2 -  \|x_{k-1}-z\|^2 \geq 
2\langle x_{k-1}-z, x_k-x_{k-1}\rangle ,
$$
 we obtain
\begin{align*} 
C &\geq  s^pk^{2p} (\Theta(x_k)-\min \Theta)+  \frac{1}{2} \left[\frac{s^{p-1}k^{2p}}{1+ \epsilon}  \|x_k-x_{k-1}\|^2  +  2 s^{\frac{p-1}{2}} \lambda_k k^p\langle  x_k-x_{k-1}, x_{k-1}-z\rangle + (1+ \epsilon)\lambda_k^2 \|x_{k-1}-z\|^2 \right]\\
& = s^p k^{2p} (\Theta(x_k)-\min \Theta)+
\frac{1}{2} \| \frac{1}{\sqrt{1 + \epsilon}} s^{\frac{p-1}{2}}k^p (x_k-x_{k-1}) +\sqrt{1 + \epsilon}\lambda_k (x_{k-1} -z) \|^2 .
\end{align*}
This implies
$$
\Theta(x_k)-\min \Theta = \mathcal O \left(   \frac{1}{k^{2p}}\right).
$$
Hence, for $\alpha \leq 3$, for all $p< \frac{2\alpha}{3} $ we have obtained
$$
\Theta(x_k)-\min \Theta = \mathcal O \left(   \frac{1}{k^{p}}\right).
$$
Since $p$ has been taken greater or equal than $\frac{1}{2}$, this corresponds to assume $\alpha \geq \frac{3}{4}$.
When $p \leq \frac{1}{2}$, on the basis of the inequality $(k+1)^{2p}-k^{2p} \leq 2p k^{2p-1}  $, we are led to take 
\begin{equation}\label{E:basic-ineq-algo2-b}
\lambda_{k}= \frac{2p}{s^{\frac{1-p}{2}}}\frac{(k-1)^{2p-1}}{k^p}\sim \frac{2p}{s^{\frac{1-p}{2}}} k^{p-1},
\end{equation}
which is an equivalent expression as in the case $p \leq \frac{1}{2}$. Since the proof is based on asymptotic equivalences, it works in a similar way, which completes the proof in the case  $\alpha \leq \frac{3}{4}$.
\end{proof}

 \begin{remark}
Our  proof uses computation based on asymptotic equivalences. This makes the proof simpler, but on the  other hand we lose some information with respect to the continuous case. We have obtained  the order of convergence $\mathcal O \left(   \frac{1}{k^{p}}\right)$ for all 
$p< \frac{2\alpha}{3} $ instead of $\frac{2\alpha}{3} $. 
It might be tempting to pass to the limit as $p \to  \frac{2\alpha}{3} $ on the inequality
$$
\mathcal{\tilde{E}}_{k,p} \leq \mathcal{\tilde{E}}_{k_0,p} \quad \mbox{ for} \quad k \geq k_0
$$
where
$$
\mathcal{\tilde{E}}_{k,p}= \ s^pk^{2p} W_k+ s^{\frac{p-1}{2}}\lambda_{k,p} k^p(h_k - h_{k-1})+ (\lambda_{k,p}^2 + \xi_{k,p})h_{k-1}  - \frac{p}{s^{1-p}}(k)^{2p-1}\|x_k-x_{k-1}\|^2
$$
and $ \lambda_{k,p}= \frac{2p}{s^{\frac{1-p}{2}}}k^{p-1}
$, 
 $\xi_{k+1,p}= - \lambda_{k+1,p}^2  - s^{\frac{p-1}{2}} (\alpha_k   \lambda_{k+1,p} (k+1)^p  -\lambda_{k,p} k^p)$.
But this would not be a correct argument: while $\mathcal{\tilde{E}}_{k,p}$ is actually continuous with respect to $p$, it is possible that at the same time, in the above inequality,  the index $k_0$ tends to $+\infty$!
However, on the basis of the continuous results, it is probable that the same order of convergence as in the continuous case is valid, that is to say $\frac{2\alpha}{3} $. 
 \end{remark} 
 
 \subsection{Convergence of iterates. The critical case $\alpha =3$}
 In the case $\alpha >3 $, as a major result, the weak convergence of the sequences generated by $\mbox{\rm(IFB)}_{\alpha}$ was obtained by
 Chambolle and Dossal \cite{CD}, and completed by Attouch-Chbani-Peypouquet-Redont  \cite{ACPR}.
 The case $\alpha =3$ corresponds  to the accelerated method of Nesterov, and   Beck-Teboulle FISTA method. In this case, the convergence of iterations is still an open problem. In line with the results obtained in the continuous case, we give  partial answers to this question. As a main result, we prove the convergence in the one-dimensional case. We therefore consider 
 $\mbox{\rm(IFB)}_{\alpha}$ with $\alpha =3$
$$\mbox{\rm(IFB)}_{3} \quad\left\{ 
\begin{array}{l}
y_k=x_k+(1 - \frac{3}{k})(x_k-x_{k-1})\\
\rule{0pt}{10pt}
x_{k+1}=\prox_{s\Psi}(y_k-s\nabla \Phi(y_k)).
\end{array}
\right.$$

\smallskip

\noindent In the case $\alpha =3$, the Lyapunov function takes the simpler form
\begin{equation}\label{algo9b}
\mathcal E (k):= s \left( k + 1\right)^2    (\Theta (x_k) - \min \Theta) + 2\| x_{k} -x^{*} + \frac{k -1}{2} ( x_{k}  - x_{k-1})  \|^2,
\end{equation}
where $x^{*} \in S$. In \cite{ACPR}, \cite{AP}, \cite{CD}, \cite{SBC} it is proved that  the sequence $(\mathcal E(k))$ is nonincreasing, and $\Theta (x_k) - \min \Theta = \mathcal O \left( \frac{1}{k^2}   \right) $.
From this we can deduce the following estimates of the sequence $(x_k)$, which are valid in a general Hilbert space.

 \begin{proposition}\label{critical-boundedness}
Let $\cH$  a  Hilbert space. Let  $(x_k)$ be a sequence  generated by algorithm  $\mbox{\rm(IFB)}_{\alpha}$  with $\alpha = 3$. Then 
$$
\sup_k \|x_k\| < + \infty  \mbox{ and }\  \sup_k  k \|x_k-x_{k-1}\| < + \infty .
$$
\end{proposition}

\begin{proof}
Since $(\mathcal E(k))$ is nonincreasing, it is bounded from above, which implies that, for some positive constant $C$
\begin{equation}\label{algo-dim1-c}
 \|  (x_{k} -x^{*})  + \frac{k -1}{2}( x_{k}  - x_{k-1}) \|^2 \leq C.
\end{equation}
 After developping the above quadratic expression , and neglecting the nonnegative term
$\|  \frac{k -1}{2}( x_{k}  - x_{k-1}) \|^2$, we obtain
$$
  \|  x_{k} -x^{*}\|^2   + (k -1)
  \left\langle  x_{k}  - x_{k-1} , x_{k} -x^{*} \right\rangle \leq C .
$$
Let us write this expression in a recursive form. Equivalently,
$$
  \|  x_{k} -x^{*}\|^2   + (k -1)
  \left\langle  (x_{k} -x^{*}) - (x_{k-1}-x^{*}) , x_{k}-x^{*} \right\rangle \leq C .
$$
Using the elementary convex inequality
$$
\|  x_{k-1} -x^{*}\|^2     \geq \|  x_{k} -x^{*}\|^2  + 2\left\langle  (x_{k-1}-x^{*}) - (x_{k} -x^{*})  , x_{k}-x^{*} \right\rangle
$$
we infer
$$
  \|  x_{k} -x^{*}\|^2   + \frac{k -1}{2}
 \|  x_{k}  - x^{*}\|^2 - \frac{k -1}{2} \|  x_{k-1} -x^{*}\|^2 \leq C .
$$
Equivalently
$$
\frac{k -1}{2}
 \|  x_{k}  - x^{*}\|^2 - \frac{k -2}{2}
 \|  x_{k-1}  - x^{*}\|^2 + \|  x_{k} -x^{*}\|^2 -   \frac{1}{2} \|  x_{k-1}  - x^{*}\|^2\leq C.
$$
Hence
$$
\left(\frac{k -1}{2}
 \|  x_{k}  - x^{*}\|^2  +  \|  x_{k} -x^{*}\|^2 \right) - \left(\frac{k -2}{2}
 \|  x_{k-1}  - x^{*}\|^2 + \|  x_{k-1}  - x^{*}\|^2 \right)\leq C.
$$
Summing these inequalities, we infer
$$
\frac{k -1}{2}
 \|  x_{k}  - x^{*}\|^2  +  \|  x_{k} -x^{*}\|^2 \leq C + Ck,
$$
which gives the announced result.
Returning to (\ref{algo-dim1-c}),  by using the triangle inequality we immediately deduce that $\sup_k  k \|x_k-x_{k-1}\| < + \infty $.
\end{proof}

 \begin{Theorem}
Take $\cH= \mathbb R$. Let us make assumptions $(H)$.
Let  $(x_k)$ be a sequence of iterates generated by algorithm $\mbox{\rm(IFB)}_{3}$ with $\alpha = 3$. Then $(x_k)$ converges, as $t\to +\infty$, to a point in $S= \argmin \Theta $.
\end{Theorem}
\begin{proof}
By Proposition  \ref{critical-boundedness},  for $\alpha =3$, the sequence $(x_k)$ is bounded, and minimizing (indeed, $\Theta (x_k) - \min \Theta = \mathcal O \left( \frac{1}{k^2}   \right) $).  As a consequence,  when $ \argmin \Theta  $ is reduced to a singleton $x^*$, it is the unique cluster point of the the sequence $(x_k)$, which implies the convergence of $(x_k)$ to $x^*$.
Thus, we just need to consider the case where 
$S=\left[a, b \right]$ is an interval with positive length.
Recall that the algorithm can be written in an equivalent way as
$$\mbox{\rm(IFB)}_{\alpha} \quad\left\{ 
\begin{array}{l}
y_k=x_k+ (1 - \frac{3}{k})(x_k-x_{k-1})\\
\rule{0pt}{10pt}
x_{k+1}=y_k-sG_s(y_k),
\end{array}
\right.$$
where the following equivalence holds true
$$
z\in \argmin \Theta\,\Longleftrightarrow\,  G_s(z)=0.
$$
Note that $y_k -x_k \to 0$, which implies that the sequences $(x_k)$  and $(y_k)$  have the same cluster points.
As for  the continuous dynamic we have to consider the following cases:

\smallskip

$\bullet$   There exists   $K \in \mathbb N$ such that
$y_k \in  \left[a, b \right]$ for all $k \geq K$. Then $G_s(y_k)=0$, which gives
$$x_{k+1}- x_k= (1 - \frac{3}{k})(x_k-x_{k-1}).$$ Solving this induction, we obtain $ 0 \leq x_{k+1}- x_k  \leq \frac {C}{k^3}$, which  implies that the sequence $(x_k )$ converges.

\smallskip

$\bullet$   There exists   $K \in \mathbb N$ such that $y_k \leq a$ for all $k \geq K$. As a consequence, $a$ is the unique cluster point of the sequence $(x_k)$  that converges to this point. Symetrically, if  $y_k \geq b$ for all $k \geq K$, then $b$ is the unique cluster point of the sequence $(x_k)$  that converges to this point.

\smallskip

$\bullet$  It remains to study the case where $a$ and $b$ are cluster points of the sequence $(y_k)$, and consequently of the sequence 
$(x_k)$. Note that $x_k-x_{k-1} \to 0$. Therefore, in parallel to the continuous case, we can find four sequences $x_{k_{n_1}}$, $x_{k_{n_2}}$, 
$x_{k_{n_3}}$, $x_{k_{n_4}}$, with $k_{n_1} < k_{n_2}< k_{n_3}< k_{n_4}$ such that, when "going from $a$ to $b$" we have
\begin{enumerate}
\item $x_{k_{n_1}} \to a$, $x_{k_{n_1}} > a$, $x_{k_{n_1} +1} - x_{k_{n_1}} > 0$;

\item $x_{k_{n_2}} \to b$, $x_{k_{n_2}} < b$, $x_{k_{n_2} +1} - x_{k_{n_2}} > 0$;

\item $x_k \in   \left[a, b \right]$ \ for all indices $k$ such that
$ k_{n_1} \leq  k \leq k_{n_2}$.
\end{enumerate}
Symetrically when "going from $b$ to $a$"
\begin{enumerate}
\item $x_{k_{n_3}} \to b$, $x_{k_{n_3}}< b$, $x_{k_{n_3} +1} - x_{k_{n_3}} < 0$;

\item $x_{k_{n_4}} \to a$, $x_{k_{n_4}} >a$, $x_{k_{n_4} +1} - x_{k_{n_4}} < 0$;

\item $x_k \in   \left[a, b \right]$ \ for all indices $k$ such that
$ k_{n_3} \leq  k \leq k_{n_4}$.
\end{enumerate}
As in the continuous  case, let us 	analyze the decay of the  quantity $|k (x_{k+1}-x_k  )|$ during a loop.
When $ k_{n_1} \leq  k \leq k_{n_2}$, we have $x_{k+1}=x_k+ (1 - \frac{3}{k})(x_k-x_{k-1})$. Equivalently
$$
(k+1) (x_{k+1}-x_k  ) - k(x_{k}-x_{k-1}  ) +3  (x_{k}-x_{k-1}  ) - (x_{k+1}-x_k  )=0.
$$
Let us sum  these equalities with $k$ varying from $ k_{n_1} $ to $ k_{n_2} -1$. We obtain
$$
k_{n_2} (x_{k_{n_2}}-x_{k_{n_2}-1}  ) -  k_{n_1} (x_{k_{n_1}}-x_{k_{n_1}-1}  ) + 3 (x_{k_{n_2}-1} - x_{k_{n_1}-1} ) -(x_{k_{n_2}} -x_{k_{n_1}})=0.
$$
 Taking account of $x_{k_{n_1}} \to a$, $x_{k_{n_2}} \to b$, $x_k-x_{k-1} \to 0$, and    of the sign of the above quantities, we deduce that,  when "going from $a$ to $b$",
\begin{equation}\label{dim1-1}
|k_{n_2} (x_{k_{n_2}}-x_{k_{n_2}-1}  )| - |k_{n_1} (x_{k_{n_1}}-x_{k_{n_1}-1}  )| \sim - 2  (b-a) .
\end{equation}
Symetrically when "going from $b$ to $a$"
\begin{equation}\label{dim1-2}
|k_{n_4} (x_{k_{n_4}}-x_{k_{n_4}-1}  )| - |k_{n_3} (x_{k_{n_3}}-x_{k_{n_3}-1}  )| \sim - 2  (b-a) .
\end{equation}
Let us now exploit the
 decreasing property of 
\begin{equation*}
\mathcal E (k):= s \left( k + 1\right)^2    (\Theta (x_k) - \min \Theta) + 2\|  x_{k} -x^{*} + \frac{k -1}{2} ( x_{k}  - x_{k-1}) \|^2.
\end{equation*}
Take  $x^{*} = x_{k_{n_2}} \in S=\left[a, b \right]$. Since $   k_{n_3} \geq k_{n_2} $, we have $\mathcal E (k_{n_3}) \leq \mathcal E (k_{n_2})  $. Equivalently,
\begin{equation*}
|  x_{k_{n_3}} -x_{k_{n_2}} + \frac{k_{n_3} -1}{2} (x_{k_{n_3}}-x_{k_{n_3} -1}) |  \leq |  \frac{k_{n_2} -1}{2} (x_{k_{n_2}}-x_{k_{n_2} -1}  ) |.
\end{equation*}
Hence
\begin{equation*}
|  (k_{n_3} -1) (x_{k_{n_3}}-x_{k_{n_3} -1}) |  \leq |  (k_{n_2} -1) (x_{k_{n_2}}-x_{k_{n_2}-1}  ) | + 2 |   x_{k_{n_3}} -x_{k_{n_2}} | . 
\end{equation*}
As a consequence
\begin{equation}\label{dim1-3}
|  k_{n_3} (x_{k_{n_3}}-x_{k_{n_3} -1}) |  \leq |  k_{n_2}  (x_{k_{n_2}}-x_{k_{n_2}-1}  ) | +  |   x_{k_{n_3}}-x_{k_{n_3} -1} | + 2 |   x_{k_{n_3}} -x_{k_{n_2}} | . 
\end{equation}
Let us now combine  (\ref{dim1-1}), (\ref{dim1-2}), (\ref{dim1-3}).
Since $x_{k_{n_2}} \to b$, $x_{k_{n_3}} \to b$, and $x_k-x_{k-1} \to 0$ , we obtain that during a loop,  when passing from $x_{k_{n_1}} \sim a$ to $x_{k_{n_4}} \sim a$, the expression $|k (x_{k}-x_{k -1} )|$
decreases by a fixed positive quantity. Hence, after a finite number of steps this expression would become negative, a clear contradiction. 
\end{proof}

\section{The perturbed case}\label{S:pert}

\subsection{Continuous dynamics}
We consider  the   equation
\begin{equation}\label{pert}
\ddot{x}(t) + \frac{\alpha}{t} \dot{x}(t) + \nabla \Phi (x(t))  =g(t)
\end{equation}
where, depending on the situation, the second member $g:[t_0 , +\infty [\rightarrow \mathcal H $ is a forcing term, or comes from the  approximation or computational errors in (\ref{edo01}). We suppose that $g$ is locally integrable to ensure existence and uniqueness for the corresponding Cauchy problem. We will show that the results of the previous sections are robust, i.e., there remain satisfied if the perturbation $g$ is not too large asymptotically.

\begin{Theorem}
Let $\mbox{\rm argmin}\Phi\neq\emptyset$, and let $x : [t_0 , +\infty [\rightarrow \mathcal H$ be a solution of (\ref{pert}) with $\alpha > 0 $. Let $p=\min(1,\frac{\alpha}{3})$ and suppose that $\int_{t_0 }^{+\infty}t^{p}g(t)dt<+\infty.$ Then
$$\Phi (x(t))-\min \Phi =   \mathcal O (\frac{1}{t^{2p}}).  $$
\end{Theorem}
\begin{proof}
Take $z\in$ argmin $\Phi$. Let us modify the energy function considering in \eqref{E:basic-Lyap} by introducing an additional term taking account of the perturbation $g$.  Let us first fix  some $T>t_0 $, and define for $t\in [t_0 ,T]$ (without ambiguity we keep the same notations as in the previous sections)
$$\mathcal{E}_{\lambda, \xi}^{p}(t)=t^{2p}\left[ \Phi(x(t))-\min \Phi\right] +\Vert \lambda(t)(x(t)-z)+t^p \dot{x}(t)\Vert^{2}+\frac{\xi(t)}{2}\Vert x(t)-z\Vert^2 +\int_{t}^{T}\langle\lambda(s)(x(s)-z)+s^{p}\dot{x}(s),s^{p}g(s) \rangle ds.$$
By derivation of $\mathcal{E}_{\lambda, \xi}^{p}(\cdot)$,  one obtains
\[
\begin{array}{lll}
\frac{d}{dt}\mathcal{E}_{\lambda, \xi}^{p}(t) &=& 2pt^{2p-1}\left[ \Phi(x(t))-\min \Phi\right] +t^{2p}\langle\nabla\Phi(x(t),\dot{x}(t)\rangle +\langle\lambda(t)(x(t)-z)+t^{p}\dot{x}(t),t^{p}(\ddot{x}(t)-g(t))\rangle\\
&+&\left[ \lambda(t)^2 +pt^{p-1}\lambda(t)+\xi(t)+t^{p}\dot{\lambda}(t)\right] \langle x(t)-z,\dot{x}(t)\rangle
+ \left( \lambda(t)\dot{\lambda}(t)+\frac{\dot{\xi}(t)}{2}\right) \Vert x(t)-z\Vert^2\\
& +& \left( t^p \lambda(t)+pt^{2p-1}\right) \Vert\dot{x}(t)\Vert^2 .
\end{array}
\]
Then note that the term $\ddot{x}(t)-g(t)$ is exactly the same as in the unperturbed case. By a similar argument,  by taking $p=\min(1,\frac{\alpha}{3})$, we obtain  $\frac{d}{dt}\mathcal{E}_{\lambda, \xi}^{p}(t)\leq 0$.

As a consequence, the energy function $\mathcal{E}_{\lambda, \xi}^{p}(\cdot)$ is nonincreasing, which implies, by definition  of $\frac{d}{dt}\mathcal{E}_{\lambda}^{p}(\cdot)$
\begin{equation}\label{pert1}
t^{2p}\left[ \Phi(x(t))-\min \Phi\right] +\Vert \lambda(t)(x(t)-z)+t^p \dot{x}(t)\Vert^{2}\leq C +\int_{t_0 }^{t}\langle\lambda(s)(x(s)-z)+s^{p}\dot{x}(s),s^{p}g(s) \rangle ds.
\end{equation}
Applying Cauchy-Schwarz inequality we deduce that
$$\Vert \lambda(t)(x(t)-z)+t^p \dot{x}(t)\Vert^{2}\leq C +\int_{t_0 }^{t}\Vert\lambda(s)(x(s)-z)+s^{p}\dot{x}(s)\Vert\Vert s^{p}g(s)\Vert ds.$$
Applying Lemma \ref{lem1}, we obtain
$$\Vert \lambda(t)(x(t)-z)+t^p \dot{x}(t)\Vert \leq\sqrt{C}+2\int_{t_0 }^{t}\Vert s^{p}g(s)\Vert ds.$$
By assumption $\int_{t_0 }^{+\infty}t^{p}g(t)dt<+\infty$. Hence 
$$\sup_{t\geq t_0 }\Vert \lambda(t)(x(t)-z)+t^p \dot{x}(t)\Vert \leq +\infty .$$
Injecting this estimate into the equation (\ref{pert1}) leads to
$$t^{2p}\left[ \Phi(x(t))-\min \Phi\right] \leq C +\left[\sqrt{C}+2\int_{t_0 }^{\infty}\Vert s^{p}g(s)\Vert ds\right]\int_{t_0 }^{\infty}\Vert s^{p}g(s)\Vert ds$$
 which immediately gives the claim.
 \end{proof}
 
 \begin{remark}
Because of the numerical importance of the subject, several recent articles have been devoted the study of perturbed versions of the  FISTA algorithm, with various formulation for the perturbations or errors: to cite a few of them Aujol-Dossal \cite{AD}, Schmidt-Le Roux-Bach \cite{SLB}, Solodov-Svaiter \cite{SoSv}, Villa-Salzo-Baldassarres-Verri \cite{VSBV}. 
As far as we know, our results are the first to address this question in  the sub-critical cas $\alpha \leq 3$.
 \end{remark}
\subsection{Algorithms}
Let us consider the perturbed algorithm
$$\mbox{\rm(IFB)}_{\alpha-pert} \quad\left\{ 
\begin{array}{l}
y_k=x_k+(1 - \frac{\alpha}{k})(x_k-x_{k-1})\\
\rule{0pt}{10pt}
x_{k+1}=\prox_{s\Psi}(y_k-s\nabla \Phi(y_k) -sg_k).
\end{array}
\right.$$
In view of the above results, and by adapting the proof of the unperturbed case, one can prove the following result, whose  detailed demonstration is left to the reader.
\begin{Theorem}\label{fast-algo-pert} 
Let us make assumptions $(H)$.
Let $(x_k)$ be a sequence of iterates generated by perturbed algorithm $\mbox{\rm(IFB)}_{\alpha, pert}$.
Suppose that $0 <\alpha <3$. Then, for any $p <\frac{2\alpha}{3}$, and $(g_k)$ such that
 $  \sum  k^p \| g_k \| < + \infty$, we have
$$
 (\Phi + \Psi) (x_k)-\min_{\mathcal H} (\Phi + \Psi) =   \mathcal O (\frac{1}{k^{p}}).
$$
\end{Theorem}

\section{Appendix}

\begin{lemma}\label{Opial} \mbox{\rm (}\cite{Op} \mbox{\rm )} Let $S$ be a nonempty subset of $\mathcal H$ and let $x:[t_0,+\infty[\to \mathcal H$. Assume that 
\begin{itemize}
\item [(i)] for every $z\in S$, $\lim_{t\to\infty}\|x(t)-z\|$ exists;
\item [(ii)] every weak sequential cluster point of $x(t)$, as $t\to\infty$, belongs to $S$.
\end{itemize}
Then $x(t)$ converges weakly as $t\to\infty$ to a point in $S$.
\end{lemma}
\begin{lemma}\label{lem1}
Let $m : [ t_0 ; T] \rightarrow [0;+\infty[ $ be integrable, and let $c > 0$. Suppose $w : [t_0 ; T] \rightarrow \mathbb R $ is continuous and
$$\frac{1}{2}w(t)^2 \leq \frac{1}{2}c^2 +\int_{t_0 }^{t}m(s)w(s)ds $$
for all $t \in [t_0 ; T]$. Then, $ \vert w(t)\vert\leq  c +\int_{t_0 }^{t}m(s)ds$ for all $t \in [t_0 ; T].$
\end{lemma}

\end{document}